\documentclass[11pt,a4paper]{article}

\usepackage[T1]{fontenc}
\usepackage[utf8]{inputenc}
\usepackage{lmodern}
\usepackage{amsmath,amssymb,amsthm,mathtools}
\usepackage{microtype}
\usepackage[margin=27mm]{geometry}
\usepackage{xcolor}
\usepackage[colorlinks=true,linkcolor=blue!55!black,citecolor=blue!55!black,
            urlcolor=blue!55!black]{hyperref}

\allowdisplaybreaks
\newtheorem{theorem}{Theorem}[section]
\newtheorem{proposition}[theorem]{Proposition}
\newtheorem{lemma}[theorem]{Lemma}
\newtheorem{corollary}[theorem]{Corollary}
\theoremstyle{definition}
\newtheorem{definition}[theorem]{Definition}
\theoremstyle{remark}

\newcommand{\Z}{\mathbb Z}
\newcommand{\R}{\mathbb R}
\newcommand{\Cc}{C_c}
\newcommand{\cE}{\mathcal E}
\newcommand{\cD}{\mathcal D}
\newcommand{\cO}{\mathcal O}
\newcommand{\abs}[1]{\lvert #1\rvert}
\newcommand{\norm}[1]{\lVert #1\rVert}
\newcommand{\lap}{\Delta_{\!d}}
\newcommand{\ind}{\mathbf 1}
\newcommand{\dd}{\mathrm d}

\title{Sharp discrete Hardy constants in dimensions three and four\\
and strict upper bounds from dimension nine}
\author{Carlos Lizama\\[2mm]
\small Departamento de Matem\'atica y Ciencia de la Computaci\'on, Facultad de Ciencia\\
\small Universidad de Santiago de Chile\\
\small Las Sophoras 173, Estaci\'on Central, Santiago, Chile\\
\small \texttt{carlos.lizama@usach.cl}}
\date{}

\begin{document}
\maketitle

\begin{abstract}
For $N\ge3$, let $C(N)$ be the optimal constant in the nearest-neighbour
Hardy inequality on $\mathbb Z^N$ with $u(0)=0$, and set
$A_N=(N-2)^2/4$. We prove that the continuum coefficient remains an upper
bound, $C(N)\le A_N$, in every dimension, and determine the exact values
$C(3)=A_3=1/4$ and $C(4)=A_4=1$. In higher dimensions we show
$C(N)<A_N$ for $N=9,10$ and obtain the explicit bound
\[
 C(N)\le 3N-\sqrt{N^2+8N-8}<2N\qquad(N\ge3),
\]
which lies below $A_N$ from dimension eleven onward. The low-dimensional
equalities follow from shifted radial supersolutions, angular convexity, and
a discrete ground-state representation. We also show that this shifted-power
mechanism cannot work at the continuum coefficient from dimension five
onward. Dimension nine is treated by a Gaussian Rayleigh--Ritz construction
combined with exact Jacobi theta-function estimates, while the higher-dimensional
bounds are obtained through finite-dimensional orbit compressions. Finally, we
derive positive spatial remainders in dimensions three and four and a spectral
consequence for the associated discrete Schr\"odinger operators in the
high-dimensional regime.
\end{abstract}



\section{Introduction}

Hardy inequalities quantify the control of concentration near a singular
point by kinetic energy.  Their classical Euclidean form states that, for
$N\ge3$ and $u\in C_c^\infty(\R^N)$,
\begin{equation}\label{eq:continuum-Hardy}
 \int_{\R^N}|\nabla u(x)|^2\,\dd x
 \ge \frac{(N-2)^2}{4}
 \int_{\R^N}\frac{|u(x)|^2}{|x|^2}\,\dd x.
\end{equation}
The coefficient $(N-2)^2/4$ is optimal.  Besides its classical role in
analysis, this number is the positivity threshold for inverse-square
Schr\"odinger forms and therefore enters stability, spectral and evolution
questions with singular potentials; see, for example,
\cite{Hardy1920,FrankSeiringer2008}.

On the lattice $\Z^N$, derivatives are replaced by nearest-neighbour
differences.  For finitely supported $u:\Z^N\to\mathbb C$ satisfying
$u(0)=0$, consider
\begin{align}
 \cE_N[u]
 &:=\sum_{j=1}^N\sum_{n\in\Z^N}
       \abs{u(n+e_j)-u(n)}^2,\label{eq:def-energy}\\
 \cD_N[u]
 &:=\sum_{n\in\Z^N\setminus\{0\}}
       \frac{\abs{u(n)}^2}{\abs{n}^2}.\label{eq:def-denominator}
\end{align}
This is the standard nearest-neighbour formulation of the multidimensional
discrete Hardy problem used, in particular, in \cite{Gupta2023}.  Its sharp
scalar constant is
\begin{equation}\label{eq:def-CN}
 C(N):=\inf_{0\ne u\in\Cc(\Z^N),\,u(0)=0}
       \frac{\cE_N[u]}{\cD_N[u]},
\end{equation}
so that
\begin{equation}\label{eq:lattice-Hardy-intro}
 \cE_N[u]\ge C(N)\cD_N[u]
\end{equation}
holds for every admissible $u$.  We write
\begin{equation}\label{eq:def-AN}
 A_N:=\frac{(N-2)^2}{4}
\end{equation}
for the continuum Hardy coefficient.

The loss of exact dilation invariance on $\Z^N$ makes the comparison between
$C(N)$ and $A_N$ nontrivial.  Large-scale lattice geometry still remembers
the Euclidean problem, but the scalar quotient in \eqref{eq:def-CN} is also
sensitive to finite-scale geometry near the singular vertex.  Thus the
natural question is whether the continuum threshold survives discretization,
that is, whether
\begin{equation}\label{eq:critical-form-intro}
 \cE_N[u]-A_N\cD_N[u]
\end{equation}
is nonnegative for every admissible $u$.

It is useful here to distinguish the scalar problem from the theory of
spatially varying Hardy weights.  One may seek an optimal function $W(n)$
such that
\[
 \cE_N[u]\ge\sum_{n\ne0}W(n)|u(n)|^2,
\]
or instead keep the prescribed weight $|n|^{-2}$ and optimize only its
scalar coefficient, as in \eqref{eq:def-CN}.  These questions are related but
not equivalent.  Ground-state representations connect sharp constants,
pointwise optimal weights and operator criticality.  In the continuum and
fractional settings this viewpoint is developed in
\cite{FrankSeiringer2008,DevyverFraasPinchover2014}; on graphs, see
\cite{KellerPinchoverPogorzelskiCMP2018,KellerPinchoverPogorzelski2020,Fischer2024}.
The recent Oberwolfach report \cite{BerchioKellerPinchoverRoncal2025}
illustrates the breadth of current activity in this direction.

The one-dimensional discrete theory is comparatively well developed.  The
classical inverse-square weight admits pointwise improvements
\cite{KellerPinchoverPogorzelski2018,KrejcirikStampach2022}, and higher-order,
weighted and fractional variants have been studied in
\cite{KellerNietschmann2023,KellerPinchoverPogorzelski2021,Gupta2024,HuangYe2024}.
Factorization and criticality methods yield improved Hardy--Rellich weights
and sharp transitions for powers of the half-line Laplacian
\cite{GerhatKrejcirikStampach2025,GerhatKrejcirikStampachRMI2025}; see also
\cite{DasFuente2026,StampachWaclawek2026}.  The multidimensional scalar
problem brings an additional angular component.

For the Euclidean lattice in several dimensions, Rozenblum--Solomyak
\cite{RozenblumSolomyak2009} obtained early spectral estimates, while
Kapitanski--Laptev \cite{KapitanskiLaptev} compared continuous and discrete
Hardy inequalities.  Keller--Lemm \cite{KellerLemm2023} identified the
continuum inverse-square asymptotics of optimal Hardy weights on $\Z^N$.
More recently, Hake--Keller--Pogorzelski
\cite[Theorem~1.1]{HakeKellerPogorzelski2026} constructed optimal Hardy
weights for fractional lattice Laplacians, including the local case; the
leading coefficient in that case is again $A_N$.  Dyda \cite{Dyda2026}
has also obtained a broad family of local and fractional lattice Hardy
inequalities with explicit, generally nonoptimal constants.  These results
explain why $A_N$ is the natural coefficient at infinity, but they do not by
themselves determine the sharp scalar in \eqref{eq:def-CN}.

Gupta \cite{Gupta2023} proved that $C(N)$ has order $N$ as $N\to\infty$ and
gave explicit linear trial upper bounds.  The sharper asymptotic results
\cite{Gupta2026,HuangYe2026} determine the leading high-dimensional
behaviour.  Those results are decisive for large $N$, whereas the fixed
low-dimensional problem asks a different question: whether the continuum
coefficient itself is still the scalar threshold.  The results of the present
paper show that $A_N$ remains an upper bound in every dimension, while a
second explicit bound of linear size becomes stronger from dimension eleven
onward.  Their comparison reveals distinct low- and high-dimensional regimes
rather than a single estimate governing all $N$.

The main result is most naturally stated by separating the dimensional
ranges in which the two upper bounds carry different information.

\begin{theorem}\label{thm:main}
For every $N\ge3$ one has
\begin{equation}\label{eq:main-universal-bound}
 C(N)\le A_N.
\end{equation}
More precisely,
\begin{equation}\label{eq:main-values}
 C(3)=A_3=\frac14,\qquad C(4)=A_4=1,
\end{equation}
while
\begin{equation}\label{eq:main-strict-9-10}
 C(N)<A_N\qquad(N=9,10).
\end{equation}
If
\[
 \Lambda_2(N):=3N-\sqrt{N^2+8N-8},
\]
then
\begin{equation}\label{eq:main-Nge11}
 C(N)\le \Lambda_2(N)<\min\{A_N,2N\}
 \qquad(N\ge11).
\end{equation}
\end{theorem}

The equalities in dimensions three and four are significant precisely
because they are not forced by the continuum limit.  In quadratic-form
language, $C(N)$ is the largest coupling $\mu$ for which
$\cE_N-\mu\cD_N$ remains nonnegative.  Thus in dimensions three and four the
nearest-neighbour lattice preserves exactly the continuum positivity
threshold for the inverse-square interaction, despite the loss of dilation
invariance.  From the point of view of discretization, this means that the
infinite lattice does not introduce a lower spurious coupling threshold in
these dimensions.  In dimensions nine and ten the continuum coefficient is
already too large, while from dimension eleven onward the explicit linear
bound in \eqref{eq:main-Nge11} gives a strictly smaller scale.  The low-dimensional forms also
admit positive spatial remainders, showing that sharpness of the scalar
coefficient is distinct from operator criticality.

A useful way to read these results is as a competition between two scales.
The geometry at infinity fixes $A_N$ as the universal continuum benchmark:
it is approached by lattice functions obtained from the continuum ground-state
profile and spreading over larger and larger annuli.  The opposite direction is controlled by finite-scale
lattice geometry.  In low dimension a shifted version of the continuum
radial ground state survives every lattice shell after the angular variables
are compressed by convexity.  In high dimension finite-dimensional orbit
compressions already lower the Rayleigh quotient below the continuum level,
and in dimension nine the complete coordinate ladder is still insufficient,
so a genuinely global radial Ritz space is needed.  The three mechanisms are
therefore different manifestations of the same question: whether the
large-scale continuum threshold can withstand the local geometry of the
lattice.

After completion of the present work, the author became aware of the recent
preprint of Alpay \cite{Alpay2026}, which also establishes the value
$C(3)=1/4$.  The proof there is different: it uses a reciprocal edge field
and an edgewise completion of squares, whereas the argument developed here is
based on shifted radial supersolutions and convex reduction of the angular
variables.  The latter framework treats dimensions three and four within the
same scheme.

The proof of Theorem~\ref{thm:main} follows this distinction in a linear
sequence of steps.  Section~2 first settles the large-scale side of the
problem.  A logarithmically truncated continuum
ground-state profile is sampled on the lattice, and a lattice--integral
comparison shows that its Rayleigh quotient tends to $A_N$.  This gives the
universal upper bound $C(N)\le A_N$.  To obtain the reverse inequality in the
dimensions where the continuum threshold survives, we then develop a discrete
ground-state representation.  For radial profiles the nearest-neighbour sum
still depends on how $|n|^2$ is distributed among the coordinates; generalized
Ces\`aro kernels yield a convexity principle that reduces this angular
dependence to the axial configuration.  Section~3 applies this reduction to
shifted continuum powers.  Their expansion at infinity identifies the smallest
shift compatible with the coefficient $A_N$.  In dimensions three and four the
endpoint shift satisfies the pointwise supersolution inequality at every
lattice site, so the ground-state representation gives the exact equalities
$C(3)=A_3$ and $C(4)=A_4$, together with positive spatial remainders.  The same
analysis also shows that from dimension five onward no member of this
shifted-power family can work at $A_N$: the condition forced at infinity is
incompatible with an inner lattice shell.

For the strict upper bounds from dimension nine onward, Section~4 turns from
supersolutions to finite-dimensional variational tests.  The Gamma--Laplace
representation of the continuum critical power motivates Gaussian lattice
profiles, while tensor factorization expresses their energy and Hardy mass
through one-dimensional theta kernels and reduces the search to a generalized
eigenvalue problem.  In dimension nine, three separated Gaussian scales
produce a trial function with quotient strictly below $A_9$.  The scales and
coefficients are selected by this Rayleigh--Ritz strategy and then replaced by
simple rational data for which one-sided theta-function estimates give a fully
exact certificate.  Section~5 develops a second variational mechanism based on
nested coordinate-support orbits under the hyperoctahedral symmetry.  The
resulting tridiagonal Ritz compressions successively provide the relevant
crossings in dimensions twelve, eleven and ten, whereas the complete
coordinate chain remains above $A_9$; this explains why the global Gaussian
construction is needed in dimension nine.  Together these two variational
mechanisms give all the strict inequalities and explicit upper bounds stated
in Theorem~\ref{thm:main}.  Section~6 translates the strict high-dimensional
inequalities into the corresponding spectral statement and records the
remaining questions.  The elementary rational arithmetic used to certify the
dimension-nine Gaussian test is deferred to Appendix~A.

\section{Preliminaries}

For nearest neighbours $m\sim n$, the discrete Laplacian is
\begin{equation}\label{eq:laplacian}
 (\lap f)(n):=\sum_{m\sim n}\bigl(f(m)-f(n)\bigr)
 =\sum_{j=1}^N\bigl(f(n+e_j)+f(n-e_j)-2f(n)\bigr).
\end{equation}
Set $X_N:=\Z^N\setminus\{0\}$.  If $E_N$ denotes the set of unordered
nearest-neighbour edges of $\Z^N$, then for the admissible functions in
\eqref{eq:def-CN} the energy \eqref{eq:def-energy} has the equivalent
unoriented-edge form
\begin{equation}\label{eq:Dirichlet-energy}
 \cE_N[u]=\sum_{\{n,m\}\in E_N}|u(n)-u(m)|^2.
\end{equation}
Thus all nearest-neighbour edges of the lattice, including those incident with
the origin, are retained in the energy.  Auxiliary profiles defined initially
on $X_N$ will be assigned their value at the origin explicitly whenever the
Laplacian or a ground-state representation is evaluated.  Throughout the
paper $A_N$ denotes the continuum coefficient \eqref{eq:def-AN}, $C(N)$ the
scalar lattice constant \eqref{eq:def-CN}, $s=|n|^2$ the squared radial
variable, and $X_N=\Z^N\setminus\{0\}$.  The symbols $\cE_N$ and $\cD_N$
always refer to the quadratic forms
\eqref{eq:def-energy}--\eqref{eq:def-denominator}.

\subsection{The continuum coefficient as an upper bound}

The first step is to realize the continuum coefficient $A_N$ by an admissible
sequence in the scalar lattice quotient.  We use a logarithmic cutoff of the
formal continuum ground-state profile and compare the resulting lattice sums
with their Euclidean integrals.  The proof is included in detail because the
uniform control of the lattice--integral error is essential for the fixed-
dimensional conclusion.

\begin{lemma}\label{lem:cube-comparison}
Let $H_R\in C^1(\R^N)$ be compactly supported for each $R$, and suppose
there is a fixed ball $B$ such that $H_R$ and $\nabla H_R$ are uniformly
bounded on $B$, while
\[
 |\nabla H_R(x)|\le C|x|^{-N-1}\qquad(x\notin B)
\]
with $C$ independent of $R$.  Then
\begin{equation}\label{eq:cube-comparison}
 \sum_{n\in\Z^N}H_R(n)=\int_{\R^N}H_R(x)\,\dd x+O(1),
\end{equation}
where the error is uniform in $R$.
\end{lemma}

\begin{proof}
Decompose $\R^N$ into unit cubes centred at lattice points.  On a cube whose
fixed enlargement does not meet $B$, the mean-value theorem bounds the
sum--integral error by a constant multiple of the supremum of
$|\nabla H_R|$ there.  The total contribution of these cubes is uniformly
bounded because
\[
 \int_2^\infty r^{N-1}r^{-N-1}\,\dd r<\infty.
\]
Only finitely many cubes meet $B$, and their contribution is uniformly
bounded by the assumed local bounds.  Compact support makes both sides finite
for each $R$.
\end{proof}

\begin{proposition}\label{prop:upper}
For every $N\ge3$,
\[
 C(N)\le A_N.
\]
\end{proposition}

\begin{proof}
Let
\[
 \alpha:=\frac{N-2}{2}.
\]
Choose smooth functions $\xi:[0,\infty)\to[0,1]$ and
$\eta:\R\to[0,1]$ such that
\[
 \xi(r)=0\quad(0\le r\le1),\qquad \xi(r)=1\quad(r\ge2),
\]
and
\[
 \eta(t)=1\quad(t\le1),\qquad \eta(t)=0\quad(t\ge2).
\]
For $R>4$ and $r>0$ set
\[
 w_R(r):=\xi(r)\eta\!\left(\frac{\log r}{\log R}\right),
 \qquad
 f_R(x):=w_R(|x|)|x|^{-\alpha}\quad(x\ne0),
\]
set $w_R(0)=0$, and put $f_R(0)=0$ after restricting to $\Z^N$.  Then $f_R$ has finite
support and is admissible in \eqref{eq:def-CN}.

\emph{Denominator.}
Since $2\alpha+2=N$,
\[
 \frac{|f_R(x)|^2}{|x|^2}=w_R(|x|)^2|x|^{-N}.
\]
The gradient of this function is $O(|x|^{-N-1})$ outside the fixed inner
annulus, uniformly in $R$.  Applying Lemma~\ref{lem:cube-comparison} and then
polar coordinates gives
\begin{align}
 \cD_N[f_R]
 &=|\mathbb S^{N-1}|\int_0^\infty \frac{w_R(r)^2}{r}\,\dd r+O(1)\notag\\
 &=|\mathbb S^{N-1}|\log R\int_0^2\eta(t)^2\,\dd t+O(1).
 \label{eq:D-asymptotic}
\end{align}
For the second equality we use $t=(\log r)/(\log R)$ on $r\ge2$; the fixed
transition region $1<r<2$ contributes $O(1)$.

\emph{Energy.}
Taylor's formula on a unit segment outside the fixed inner annulus gives
\[
 f_R(n+e_j)-f_R(n)=\partial_j f_R(n)+R_{n,j},
 \qquad
 |R_{n,j}|\le C\sup_{[n,n+e_j]}|D^2f_R|.
\]
Uniformly in $R$,
\[
 |\nabla f_R(x)|\le C|x|^{-\alpha-1},
 \qquad
 |D^2f_R(x)|\le C|x|^{-\alpha-2}
\]
outside that annulus.  Hence the total contribution of the cross terms is
bounded by a constant multiple of
\[
 \int_2^\infty r^{N-1}r^{-\alpha-1}r^{-\alpha-2}\,\dd r
 =\int_2^\infty r^{-2}\,\dd r,
\]
while the squared remainders are bounded by a constant multiple of
$\int_2^\infty r^{-3}\,\dd r$.  Therefore
\begin{equation}\label{eq:E-gradient-sum}
 \cE_N[f_R]=\sum_{n\in\Z^N}|\nabla f_R(n)|^2+O(1).
\end{equation}
Furthermore,
$|\nabla(|\nabla f_R|^2)|\le C|x|^{-N-1}$ outside a fixed ball, so
Lemma~\ref{lem:cube-comparison} applied to $H_R=|\nabla f_R|^2$ yields
\begin{equation}\label{eq:E-gradient-integral}
 \cE_N[f_R]=\int_{\R^N}|\nabla f_R(x)|^2\,\dd x+O(1).
\end{equation}

It remains to evaluate this continuum integral.  Since $f_R$ is radial,
\[
 \partial_r\bigl(w_R(r)r^{-\alpha}\bigr)
 =w_R'(r)r^{-\alpha}-\alpha w_R(r)r^{-\alpha-1}.
\]
Using $2\alpha=N-2$, polar coordinates give
\begin{align*}
 \int_{\R^N}|\nabla f_R|^2\,\dd x
 =|\mathbb S^{N-1}|\int_0^\infty
 \left(\alpha^2\frac{w_R(r)^2}{r}
       -2\alpha w_R(r)w_R'(r)+r|w_R'(r)|^2\right)\,\dd r.
\end{align*}
The middle term vanishes exactly because
\[
 \int_0^\infty w_R(r)w_R'(r)\,\dd r
 =\frac12[w_R(r)^2]_{0}^{\infty}=0.
\]
The last integral is $O(1)$: the derivative of $\xi$ is supported in the
fixed annulus $1<r<2$, while on the logarithmic cutoff region
$|w_R'(r)|\le C/(r\log R)$ and hence its contribution is
$O((\log R)^{-1})$.  Combining this with
\eqref{eq:E-gradient-integral} and the same radial integral already appearing
in \eqref{eq:D-asymptotic}, we obtain
\[
 \cE_N[f_R]
 =\alpha^2|\mathbb S^{N-1}|\log R
   \int_0^2\eta(t)^2\,\dd t+O(1).
\]
\emph{Conclusion.}
The numerator and denominator therefore have the same positive leading factor,
and
\[
 \lim_{R\to\infty}\frac{\cE_N[f_R]}{\cD_N[f_R]}
 =\alpha^2=\frac{(N-2)^2}{4}=A_N.
\]
The variational definition \eqref{eq:def-CN} gives $C(N)\le A_N$.
\end{proof}

\subsection{Ground-state representation}

We use the unoriented-edge realization of the Dirichlet energy in
\eqref{eq:Dirichlet-energy}.  The next identity is the discrete $p=2$
ground-state transform adapted to this realization.  It converts a positive
pointwise supersolution into a global Hardy inequality and simultaneously
retains the nonnegative remainder term.  General continuum and graph versions,
together with their role in criticality theory, may be found in
\cite{FrankSeiringer2008,KellerPinchoverPogorzelskiCMP2018,Fischer2024}.

\begin{lemma}\label{lem:groundstate}
Let $h>0$ on $\Z^N\setminus\{0\}$ and set $h(0)=0$.  For every finitely
supported $u$ with $u(0)=0$,
\begin{align}
 \cE_N[u]
 &-\sum_{n\ne0}\frac{(-\lap h)(n)}{h(n)}|u(n)|^2\notag\\
 &=\sum_{\substack{\{n,m\}\in E_N\\n,m\ne0}}
 h(n)h(m)
 \left|\frac{u(n)}{h(n)}-\frac{u(m)}{h(m)}\right|^2.\label{eq:gsr}
\end{align}
In particular, if
\begin{equation}\label{eq:supersolution}
 |n|^2\frac{(-\lap h)(n)}{h(n)}\ge\lambda
 \qquad(n\ne0),
\end{equation}
then $C(N)\ge\lambda$.
\end{lemma}

\begin{proof}
Write $u(n)=h(n)v(n)$ for $n\ne0$.  On an edge with nonzero endpoints,
direct expansion gives
\begin{align*}
 |h(n)v(n)-h(m)v(m)|^2
 &-h(n)h(m)|v(n)-v(m)|^2\\
 &=(h(n)-h(m))
   \bigl(h(n)|v(n)|^2-h(m)|v(m)|^2\bigr).
\end{align*}
After summing over all such edges, the coefficient of $|v(n)|^2$ is
\[
 h(n)\sum_{\substack{m\sim n\\m\ne0}}\bigl(h(n)-h(m)\bigr).
\]
An edge incident with the origin contributes
$|u(n)|^2=h(n)^2|v(n)|^2$ to the energy and no term to the transformed edge
sum.  Adding these boundary contributions gives the coefficient
\[
 h(n)\left[\sum_{\substack{m\sim n\\m\ne0}}
 (h(n)-h(m))+\ind_{\{|n|=1\}}h(n)\right]
 =h(n)(-\lap h)(n),
\]
where the last equality uses $h(0)=0$.  Since
$h(n)(-\lap h)(n)|v(n)|^2=(-\lap h)(n)|u(n)|^2/h(n)$,
this proves \eqref{eq:gsr}.  All sums involving $u$ or $v$ are finite,
including edges leaving their support, so no convergence argument is needed.
If \eqref{eq:supersolution} holds, the potential term in \eqref{eq:gsr} is at
least $\lambda\cD_N[u]$ and the remainder is nonnegative.
\end{proof}

\subsection{A convex reduction of the angular variables}

For $\alpha>0$ and $m\in\mathbb N_0$, set
\[
 k^\alpha(m):=\frac{\Gamma(\alpha+m)}
 {\Gamma(\alpha)\Gamma(m+1)}.
\]
These are the generalized Ces\`aro kernels.  Their recurrence, generating
function, convolution properties, monotonicity, and use in discrete
fractional calculus are developed in
\cite{AbadiasMiana2018,GoodrichLizama2020}.  In particular, the binomial
formula gives
\begin{equation}\label{eq:cesaro}
 \sum_{m=0}^\infty k^\alpha(m)z^m=(1-z)^{-\alpha}\quad(|z|<1),
 \qquad
 k^\alpha(m+1)=\frac{m+\alpha}{m+1}k^\alpha(m).
\end{equation}
In particular, $k^\alpha(m)$ is strictly decreasing in $m$ when
$0<\alpha<1$.  The even part of the series is
\begin{equation}\label{eq:even-kernel}
 \mathcal K_\alpha(z)
 :=\sum_{m=0}^\infty k^\alpha(2m)z^m
 =\frac12\left[(1-\sqrt z)^{-\alpha}
                +(1+\sqrt z)^{-\alpha}\right],\qquad 0\le z<1.
\end{equation}

The nearest-neighbour operator has an exact radial form that will be used
throughout the paper.

\begin{lemma}[Radial nearest-neighbour identity]\label{lem:radial-neighbour}
Let $H:\mathbb N_0\to\mathbb R$, let $n\in\Z^N$, and put $s=|n|^2$.  Define
\begin{equation}\label{eq:coordinate-multiplicities}
 \mu_r(n):=\#\{j\in\{1,\ldots,N\}: |n_j|=r\},\qquad r\in\mathbb N_0.
\end{equation}
Then
\begin{equation}\label{eq:radial-neighbour}
 \sum_{m\sim n}H(|m|^2)
 =2\mu_0(n)H(s+1)
 +\sum_{r\ge1}\mu_r(n)
   \bigl[H(s+1-2r)+H(s+1+2r)\bigr].
\end{equation}
\end{lemma}

\begin{proof}
For a fixed coordinate $j$,
\[
 |n\pm e_j|^2=|n|^2+1\pm2n_j=s+1\pm2n_j.
\]
If $n_j=0$, the two neighbours therefore contribute $2H(s+1)$.  If
$|n_j|=r\ge1$, their contribution is
$H(s+1-2r)+H(s+1+2r)$, independently of the sign of $n_j$.  Grouping the
coordinates according to their absolute values gives
\eqref{eq:radial-neighbour}.
\end{proof}

The multiplicities \eqref{eq:coordinate-multiplicities} are invariants of the
signed-permutation action on $\Z^N$.  They encode the radial action of the
adjacency operator and will also enter the orbit reduction in
Section~\ref{sec:high-dimensional-orbits}.

The neighbour sum of a radial profile still depends on the distribution of
the coordinate multiplicities.  The next lemma removes this angular
dependence by showing that, at fixed $|n|^2$, concentrating all squared
coordinate mass in one coordinate gives the largest possible neighbour sum.

The positivity of the coefficients in \eqref{eq:even-kernel} is the essential
input.  It converts the angular problem into a one-variable convexity argument
and will allow the supersolution inequalities to be checked only in the axial
model.  The resulting principle is useful beyond the particular shifted
powers used below: any positive linear combination of radial powers of the
form appearing in Lemma~\ref{lem:angular} is maximized, at fixed $|n|^2$, by
concentrating the squared coordinate mass on one axis.  We retain the
Ces\`aro-kernel formulation because it makes the positivity mechanism
explicit and stable under such positive combinations.

\begin{lemma}\label{lem:angular}
Let
\[
 F(q)=\sum_{r=1}^M c_r(q+\beta)^{-p_r},
 \qquad c_r,p_r,\beta>0.
\]
For $n\in\Z^N$ and $s=|n|^2$,
\begin{align}
 &\sum_{j=1}^N\bigl(F(s+1+2n_j)+F(s+1-2n_j)\bigr)\notag\\
 &\quad\le F(s+1+2\sqrt s)+F(s+1-2\sqrt s)
              +2(N-1)F(s+1).\label{eq:axial}
\end{align}
\end{lemma}

\begin{proof}
Linearity reduces the proof to $F(q)=(q+\beta)^{-p}$.  Put
$a=s+1+\beta$.  Since $a>2\sqrt{s}$, formula \eqref{eq:even-kernel} gives,
for $0\le t\le s$,
\begin{align*}
 \Phi(t)&:=F(s+1+2\sqrt t)+F(s+1-2\sqrt t)\\
 &=2a^{-p}\sum_{m=0}^\infty k^p(2m)
       \left(\frac{4t}{a^2}\right)^m.
\end{align*}
All coefficients are nonnegative, so $\Phi$ is convex on $[0,s]$.  For
$x,y\ge0$ with $x+y\le s$, convexity implies
\[
 \Phi(x)+\Phi(y)\le\Phi(x+y)+\Phi(0).
\]
For example, this follows by integrating the increasing derivative of
$\Phi$; the general case follows by approximation if an endpoint is involved.
Applying this inequality repeatedly to $t_j=n_j^2$, for which
$\sum_jt_j=s$, yields
\[
 \sum_{j=1}^N\Phi(t_j)\le\Phi(s)+(N-1)\Phi(0),
\]
which is \eqref{eq:axial}.
\end{proof}

\section{Dimensions three and four}

The continuum formal ground state is $|x|^{-(N-2)/2}$.  In the squared radial
variable this is $s^{-(N-2)/4}$, so the simplest lattice perturbation that
preserves the continuum decay while allowing the first lattice correction to
adjust is a shift $s\mapsto s+\beta$.  This leads to the family below.  The
point of Proposition~\ref{prop:critical-shift} is that the shift is not a
free fitting parameter: its endpoint value is forced by the asymptotic
nearest-neighbour expansion.

Put
\begin{equation}\label{eq:pN-hbeta}
 p_N:=\frac{N-2}{4},
 \qquad
 h_{N,\beta}(n):=(|n|^2+\beta)^{-p_N}\quad(n\ne0),
\end{equation}
and set $h_{N,\beta}(0)=0$.  If $s=|n|^2\ge2$, no neighbour of $n$ is the
origin.  Lemma~\ref{lem:angular} shows that the sum of the neighbouring values,
divided by $h_{N,\beta}(n)$, is bounded above by
\begin{align}
 \mathcal R_{N,\beta}(s):={}&
 \left(\frac{s+\beta}{s+1+\beta+2\sqrt s}\right)^{p_N}
 +\left(\frac{s+\beta}{s+1+\beta-2\sqrt s}\right)^{p_N}\notag\\
 &+2(N-1)\left(\frac{s+\beta}{s+1+\beta}\right)^{p_N}.
 \label{eq:axial-ratio}
\end{align}
Therefore the supersolution inequality at coefficient $A_N$ follows if
\begin{equation}\label{eq:critical-axial}
 \mathcal R_{N,\beta}(s)\le2N-\frac{A_N}{s}.
\end{equation}

\subsection{The necessary shift at infinity}

Before checking the supersolution inequality at individual shells, one must
determine which shifts are even compatible with the continuum asymptotic
coefficient.  Expanding the axial neighbour quotient supplies this necessary
condition and identifies a distinguished value $\beta_N^*$.

The conclusion is only a condition at infinity: satisfying it does not ensure
that the finitely many inner shells obey the supersolution inequality.  This
distinction is precisely what separates dimensions three and four from the
higher-dimensional cases.

For the unshifted power, Huang--Ye
\cite[Theorem~6.1 and Remark~6.2]{HuangYe2024} compute the angular correction
to the lattice Laplacian.  On an axis their expansion reads
\[
 \frac{-\lap h_{N,0}}{h_{N,0}}
 =\frac{A_N}{s}
 -\frac{(N^2-4)(N^2+16N-12)}{192s^2}+O(s^{-3}).
\]
The $s^{-2}=|n|^{-4}$ correction for the unshifted radial power is therefore
already contained in \cite{HuangYe2024}.  Proposition~\ref{prop:critical-shift}
determines how the shift parameter $\beta$ modifies this coefficient: its
cancellation singles out the threshold $\beta_N^*$.  The subsequent lemmas
then supply the global information needed on the finite shells.

\begin{proposition}\label{prop:critical-shift}
As $s\to\infty$,
\begin{equation}\label{eq:affine-expansion}
 \mathcal R_{N,\beta}(s)
 =2N-\frac{(N-2)^2}{4s}
 +\frac{(N-2)(\beta_N^*-\beta)}{s^2}+O(s^{-3}),
\end{equation}
where
\begin{equation}\label{eq:beta-star}
 \beta_N^*=\frac{N^3+18N^2+20N-24}{192}.
\end{equation}
If $\beta<\beta_N^*$, then $h_{N,\beta}$ cannot be a global supersolution for
the potential $A_N|n|^{-2}$.
\end{proposition}

\begin{proof}
Let $r=s^{-1/2}$, $p=p_N$, and $b=\beta+1$.  Formula
\eqref{eq:axial-ratio} becomes
\begin{align}
 \mathcal R_{N,\beta}(s)
 =(1+\beta r^2)^p\bigl[&(1+2r+br^2)^{-p}
 +(1-2r+br^2)^{-p}\notag\\
 &+2(N-1)(1+br^2)^{-p}\bigr].\label{eq:R-r-form}
\end{align}
Formula \eqref{eq:cesaro} may be used at this point in the equivalent form
\begin{equation}\label{eq:cesaro-taylor}
 (1+x)^{-p}=\sum_{m=0}^\infty(-1)^m k^p(m)x^m,
 \qquad |x|<1.
\end{equation}
Thus the Taylor coefficients below are precisely the generalized Ces\`aro
coefficients $k^p(m)$ with alternating signs.  Since only the terms through
order $r^4$ are required, displaying those four coefficients is shorter than
introducing coefficient convolutions for the expressions
$x=\pm2r+br^2$.  We therefore use
\[
 (1+x)^{-p}=1-px+\frac{p(p+1)}2x^2
 -\frac{p(p+1)(p+2)}6x^3
 +\frac{p(p+1)(p+2)(p+3)}{24}x^4+O(x^5),
\]
and adding the expressions with signs $+$ and $-$, the bracket in
\eqref{eq:R-r-form} equals
\[
 2N+B_2r^2+B_4r^4+O(r^6),
\]
where
\begin{align*}
 B_2&=-2Npb+4p(p+1),\\
 B_4&=Np(p+1)b^2-4p(p+1)(p+2)b\\
 &\hspace{25mm}+\frac43p(p+1)(p+2)(p+3).
\end{align*}
The bracket is an even analytic function of $r$ near the origin; hence the
first term not displayed is of order $r^6$.
On the other hand,
\[
 (1+\beta r^2)^p
 =1+p\beta r^2+\frac{p(p-1)}2\beta^2r^4+O(r^6).
\]
Thus the coefficients of $r^2$ and $r^4$ in
\eqref{eq:R-r-form} are, respectively,
\begin{align*}
 B_2+2Np\beta
 &=2p(2p+2-N),\\
 B_4+p\beta B_2+Np(p-1)\beta^2.
\end{align*}
Substitution of $p=(N-2)/4$ simplifies these expressions to
\[
 -\frac{(N-2)^2}{4},
 \qquad
 \frac{N-2}{192}
 \bigl(N^3+18N^2+20N-24-192\beta\bigr).
\]
Since $r^2=s^{-1}$, this is \eqref{eq:affine-expansion}.

If $\beta<\beta_N^*$, the $s^{-2}$ coefficient is positive.  At an axial lattice point $n=(m,0,\ldots,0)$ one has
$\mu_0=N-1$ and $\mu_m=1$, so \eqref{eq:radial-neighbour} shows that
\eqref{eq:axial-ratio} is the exact neighbour quotient.  Hence
\eqref{eq:critical-axial} fails for all sufficiently large integers $m$,
proving the final statement.
\end{proof}

The preceding proposition leaves open whether a shift
$\beta\ge\beta_N^*$ can satisfy all finite-shell inequalities.  The following
observation shows that the simple shifted power \eqref{eq:pN-hbeta} reaches its
limit in dimension four: from dimension five onward, an explicit inner shell
is incompatible with the condition forced by infinity.

This obstruction is pointwise and exact.  It also explains why the treatment
of dimensions five through eight requires profiles with additional radial or
angular structure.

\begin{proposition}\label{prop:shifted-power-fails}
Let $N\ge5$.  There is no $\beta>0$ for which the shifted radial profile
\[
 h_{N,\beta}(n)=(|n|^2+\beta)^{-(N-2)/4}\quad(n\ne0),
 \qquad h_{N,\beta}(0)=0,
\]
is a global supersolution at the continuum coefficient $A_N$, that is,
\[
 -\lap h_{N,\beta}(n)\ge \frac{A_N}{|n|^2}h_{N,\beta}(n)
 \qquad(n\ne0).
\]
\end{proposition}

\begin{proof}
Proposition~\ref{prop:critical-shift} shows first that any such supersolution
would necessarily satisfy $\beta\ge\beta_N^*$.

Consider first $N=5$, for which $p_5=3/4$ and
$\beta_5^*=217/64$.  At $n=(1,1,0,0,0)$ one has
$s=2$, $\mu_0=3$, and $\mu_1=2$.  Hence
\eqref{eq:radial-neighbour} gives the neighbour quotient in the form
\begin{equation}\label{eq:Q5-general}
 Q_5(\beta)=
 2\left(\frac{\beta+2}{\beta+1}\right)^{3/4}
 +6\left(\frac{\beta+2}{\beta+3}\right)^{3/4}
 +2\left(\frac{\beta+2}{\beta+5}\right)^{3/4}.
\end{equation}
At the smallest admissible shift,
\[
 Q_5(\beta_5^*)=
 2\left(\frac{345}{281}\right)^{3/4}
 +6\left(\frac{345}{409}\right)^{3/4}
 +2\left(\frac{345}{537}\right)^{3/4}.
\]
The elementary comparisons
\[
 \left(\frac{345}{281}\right)^{3/4}>\frac{23}{20},\qquad
 \left(\frac{345}{409}\right)^{3/4}>\frac78,\qquad
 \left(\frac{345}{537}\right)^{3/4}>\frac7{10}
\]
follow after raising each positive inequality to the fourth power.  Hence
\[
 Q_5(\beta_5^*)>\frac{179}{20}>\frac{71}{8}
 =2N-\frac{A_N}{|n|^2}.
\]
Moreover, $Q_5(\beta)$ is increasing for $\beta\ge\beta_5^*$.  Differentiating
\eqref{eq:Q5-general}, the only negative contribution comes from the first
term and is dominated by the derivative of the middle term because
\[
 3\left(\frac{\beta+1}{\beta+3}\right)^{7/4}>1
 \qquad(\beta\ge\beta_5^*>3).
\]
The last inequality follows from $3(2/3)^{7/4}>1$.  Thus the shell
$|n|^2=2$ rules out every $\beta\ge\beta_5^*$.

For $N\ge6$, take $n=e_1$.  With $H(0)=0$ and
$H(q)=(q+\beta)^{-p_N}$ for $q\ge1$, one has
$\mu_0=N-1$ and $\mu_1=1$, so \eqref{eq:radial-neighbour} gives
\[
 Q_N(\beta)=
 \left(\frac{\beta+1}{\beta+4}\right)^{p_N}
 +2(N-1)\left(\frac{\beta+1}{\beta+2}\right)^{p_N}.
\]
This function is increasing in $\beta$.  Moreover,
\begin{equation}\label{eq:beta-star-simple-bound}
 \beta_N^*-(N-1)
 =\frac{(N-6)(N^2+24N-28)}{192}\ge0
 \qquad(N\ge6).
\end{equation}
Since $p_N\ge1$, Bernoulli's inequality and
\eqref{eq:beta-star-simple-bound} give, for every $\beta\ge\beta_N^*$,
\begin{align*}
 Q_N(\beta)
 &\ge2N-1-\frac{3p_N}{\beta+4}
             -\frac{2(N-1)p_N}{\beta+2}\\
 &\ge2N-1-\frac{3(N-2)}{4(N+3)}
             -\frac{(N-1)(N-2)}{2(N+1)}.
\end{align*}
The correction subtracted from $2N$ is strictly smaller than
\[
 1+\frac34+\frac{N-2}{2}=\frac N2+\frac34,
\]
whereas
\[
 A_N-\left(\frac N2+\frac34\right)
 =\frac{N^2-6N+1}{4}>0\qquad(N\ge6).
\]
Thus $Q_N(\beta)>2N-A_N$, and the supersolution inequality already fails at
the first shell for every $N\ge6$.
\end{proof}

For the dimensions of interest, \eqref{eq:beta-star} gives
\begin{equation}\label{eq:beta34}
 \beta_3^*=\frac{75}{64},
 \qquad
 \beta_4^*=\frac{17}{8}.
\end{equation}

\subsection{The three-dimensional supersolution}

The necessary shift in dimension three is now fixed by
\eqref{eq:beta34}.  It remains to prove that this value controls every shell,
not merely the asymptotic tail.  The angular reduction turns that global task
into a scalar inequality in $s=|n|^2$.

The next lemma proves the required pointwise estimate.  Its proof keeps all
rounding directions explicit and treats the shell adjacent to the origin
separately.

\begin{lemma}\label{lem:N3}
Let
\[
 h_3(0)=0,
 \qquad
 h_3(n)=\left(|n|^2+\frac{75}{64}\right)^{-1/4}
 \quad(n\ne0).
\]
Then
\begin{equation}\label{eq:N3-super}
 -\lap h_3(n)>\frac{1}{4|n|^2}h_3(n)
 \qquad(n\ne0).
\end{equation}
\end{lemma}

\begin{proof}
Let $s=|n|^2\ge2$.  By Lemma~\ref{lem:angular}, it is enough to prove the
axial inequality \eqref{eq:critical-axial}.  Set
\begin{equation}\label{eq:N3-variables}
 u:=\frac1{64s+139},
 \qquad z:=256u(1-139u).
\end{equation}
Since $s\ge2$,
\[
 0<u\le\frac1{267},\qquad 0<z<1.
\]
Moreover,
\[
 \frac{s+75/64}{s+139/64}=1-64u,
 \qquad
 \frac{4s}{(s+139/64)^2}=z.
\]
Using \eqref{eq:even-kernel} in the two axial terms of
\eqref{eq:axial-ratio}, one obtains the exact identities
\begin{align}
 \frac12\mathcal R_{3,75/64}(s)
 &=(1-64u)^{1/4}\bigl(2+\mathcal K_{1/4}(z)\bigr),
 \label{eq:N3-R-transform}\\
 \frac12\left(6-\frac1{4s}\right)
 &=3-\frac{8u}{1-139u}.\label{eq:N3-critical-transform}
\end{align}

The binomial expansion of $(1-x)^{1/4}$ has negative coefficients from the
linear term onward.  Hence, for $0\le x<1$,
\[
 (1-x)^{1/4}\le1-\frac{x}{4}-\frac{3x^2}{32},
\]
and therefore
\begin{equation}\label{eq:N3-binomial-bound}
 (1-64u)^{1/4}\le1-16u-384u^2.
\end{equation}
By \eqref{eq:cesaro}, the coefficients $k^{1/4}(m)$ decrease.  Since
\[
 k^{1/4}(2)=\frac5{32},\qquad
 k^{1/4}(4)=\frac{195}{2048},
\]
the tail of \eqref{eq:even-kernel} is bounded by a geometric series:
\begin{equation}\label{eq:N3-kernel-bound}
 \mathcal K_{1/4}(z)
 \le1+\frac5{32}z+\frac{195}{2048}\frac{z^2}{1-z}.
\end{equation}
Both factors on the right sides of \eqref{eq:N3-binomial-bound} and
\eqref{eq:N3-kernel-bound} are positive in the stated range.  Substitution of
\eqref{eq:N3-variables} gives
\begin{align}
 &3-\frac{8u}{1-139u}
 -(1-16u-384u^2)
 \left(3+\frac5{32}z+\frac{195}{2048}\frac{z^2}{1-z}\right)\notag\\
 &\hspace{18mm}=
 \frac{8u^3P_3(u)}{(1-139u)(1-256u+35584u^2)},\label{eq:N3-rational-identity}
\end{align}
where
\begin{equation}\label{eq:P3}
 P_3(u)=1119+u(13708556-2401602524u)
       +u^3(10356056000+515638848000u).
\end{equation}
The denominators in \eqref{eq:N3-rational-identity} are positive because
$1-139u>0$ and $1-256u+35584u^2=1-z>0$.  Furthermore,
\[
 13708556-2401602524u
 \ge13708556-\frac{2401602524}{267}
 =\frac{1258581928}{267}>0.
\]
Thus $P_3(u)>0$, so \eqref{eq:N3-R-transform}--\eqref{eq:N3-rational-identity}
prove the strict axial inequality for $s\ge2$.

It remains to consider $s=1$.  Up to a lattice symmetry, $n=e_1$.
Write $F_3(q)=(q+75/64)^{-1/4}$ for $q\ge1$ and set $H_3(0)=0$,
$H_3(q)=F_3(q)$ for $q\ge1$.  Since $\mu_0(e_1)=2$ and
$\mu_1(e_1)=1$, formula \eqref{eq:radial-neighbour} gives
\[
 \sum_{m\sim e_1}h_3(m)=F_3(4)+4F_3(2)<5F_3(1)
 <\frac{23}{4}F_3(1),
\]
where the strict inequalities use the monotonicity of $F_3$.
Equivalently, $-\lap h_3(e_1)>\tfrac14h_3(e_1)$, which completes the proof.
\end{proof}

\subsection{The four-dimensional supersolution}

In dimension four the exponent is $1/2$, and the same strategy leads to
slightly simpler binomial coefficients.  The shift $17/8$ is again forced by
the expansion at infinity, while the remaining issue is the sign of the
finite-shell residual.

The following lemma verifies this sign uniformly and then checks the first
shell directly.  Together with the ground-state representation it supplies the
matching lower bound for $C(4)$.

\begin{lemma}\label{lem:N4}
Let
\[
 h_4(0)=0,
 \qquad
 h_4(n)=\left(|n|^2+\frac{17}{8}\right)^{-1/2}
 \quad(n\ne0).
\]
Then
\begin{equation}\label{eq:N4-super}
 -\lap h_4(n)>\frac{1}{|n|^2}h_4(n)
 \qquad(n\ne0).
\end{equation}
\end{lemma}

\begin{proof}
Let $s=|n|^2\ge2$ and use Lemma~\ref{lem:angular}.  Put
\begin{equation}\label{eq:N4-variables}
 u:=\frac1{8s+25},
 \qquad z:=32u(1-25u).
\end{equation}
Then
\[
 0<u\le\frac1{41},\qquad 0<z<1,
\]
and direct substitution in \eqref{eq:axial-ratio} gives
\begin{align}
 \frac12\mathcal R_{4,17/8}(s)
 &=\sqrt{1-8u}\bigl(3+\mathcal K_{1/2}(z)\bigr),
 \label{eq:N4-R-transform}\\
 \frac12\left(8-\frac1s\right)
 &=4\frac{1-26u}{1-25u}.\label{eq:N4-critical-transform}
\end{align}
The binomial expansion gives
\begin{equation}\label{eq:N4-binomial-bound}
 \sqrt{1-8u}\le1-4u-8u^2.
\end{equation}
The polynomial on the right is positive because
$1-4u-8u^2\ge1-4/41-8/41^2>0$.
Also, $k^{1/2}(m)$ decreases and
\[
 k^{1/2}(2)=\frac38,
 \qquad k^{1/2}(4)=\frac{35}{128}.
\]
Consequently,
\begin{equation}\label{eq:N4-kernel-bound}
 \mathcal K_{1/2}(z)
 \le1+\frac38z+\frac{35}{128}\frac{z^2}{1-z}.
\end{equation}
After inserting \eqref{eq:N4-variables}, the difference between the right
side of \eqref{eq:N4-critical-transform} and the product of the upper bounds
in \eqref{eq:N4-binomial-bound}--\eqref{eq:N4-kernel-bound} is
\begin{align}
 &4\frac{1-26u}{1-25u}
 -(1-4u-8u^2)
 \left(4+\frac38z+\frac{35}{128}\frac{z^2}{1-z}\right)\notag\\
 &\hspace{18mm}=
 \frac{4u^3P_4(u)}{(1-25u)(1-32u+800u^2)},\label{eq:N4-rational-identity}
\end{align}
where
\begin{equation}\label{eq:P4}
 P_4(u)=639-u(4558+730650u)
       +u^3(1235000+3250000u).
\end{equation}
The two denominator factors are positive, the second being $1-z$.  Dropping
the last positive term in \eqref{eq:P4} and using $u\le1/41$ gives
\[
 P_4(u)\ge639-\frac{4558}{41}-\frac{730650}{41^2}
 =\frac{156631}{1681}>0.
\]
This proves the strict supersolution inequality for $s\ge2$.

For $s=1$, take $n=e_1$.  Define $F_4(q)=(q+17/8)^{-1/2}$ for
$q\ge1$ and $H_4(0)=0$, $H_4(q)=F_4(q)$ for $q\ge1$.
Since $\mu_0(e_1)=3$ and $\mu_1(e_1)=1$, formula
\eqref{eq:radial-neighbour} yields
\[
 \sum_{m\sim e_1}h_4(m)=F_4(4)+6F_4(2)<7F_4(1),
\]
by strict monotonicity of $F_4$.  This is equivalent to
$-\lap h_4(e_1)>h_4(e_1)$.
\end{proof}

The pointwise estimates now give the exact low-dimensional constants, while
the asymptotic obstruction identifies the endpoint shifts within the chosen
radial family.

\begin{proposition}\label{prop:low-dimensional-values}
One has
\[
 C(3)=\frac14,\qquad C(4)=1.
\]
\end{proposition}

\begin{proof}
Lemma~\ref{lem:groundstate} applied to Lemmas~\ref{lem:N3} and
\ref{lem:N4} yields $C(3)\ge1/4$ and $C(4)\ge1$.  Proposition~\ref{prop:upper}
gives the reverse inequalities.
\end{proof}

The shifts appearing in Lemmas~\ref{lem:N3} and \ref{lem:N4} have an
intrinsic endpoint characterization.  Proposition~\ref{prop:critical-shift}
shows that the
behaviour at infinity forces $\beta\ge\beta_N^*$, while the two lemmas show
that in dimensions three and four the endpoint value itself already controls
every finite shell.  Thus these are the smallest perturbations of the
continuum radial power within the family \eqref{eq:pN-hbeta}; no optimality
among arbitrary supersolutions is asserted.

\begin{proposition}\label{prop:minimal-shifts}
Among the functions
\[
 h_{N,\beta}(n)=(|n|^2+\beta)^{-(N-2)/4}\quad(n\ne0),
 \qquad h_{N,\beta}(0)=0,\quad\beta>0,
\]
the smallest shifts that satisfy
$-\lap h_{N,\beta}(n)\ge A_N|n|^{-2}h_{N,\beta}(n)$ at every $n\ne0$
in dimensions three and four are
\begin{equation}\label{eq:optimal-shifts-intro}
 \beta_3^{\rm opt}=\frac{75}{64},
 \qquad
 \beta_4^{\rm opt}=\frac{17}{8}.
\end{equation}
\end{proposition}

\begin{proof}
Proposition~\ref{prop:critical-shift}, together with \eqref{eq:beta34}, shows
that no smaller shift can work.  Lemmas~\ref{lem:N3} and \ref{lem:N4} show
that the two displayed endpoint shifts do produce global supersolutions.
\end{proof}

\subsection{Sharp scalar constants and positive remainders}

Sharpness of a scalar coefficient only excludes a uniform increase of that
coefficient while the weight is fixed.  It does not exclude an additional
spatially varying positive weight.  The strict pointwise inequalities proved
above make this distinction concrete.

For a nonnegative graph form $Q$, we use the standard terminology that $Q$
is subcritical if there is a nonzero nonnegative weight $V$ with
$Q[u]\ge\sum_n V(n)|u(n)|^2$ for every finitely supported $u$; otherwise it
is critical.  See \cite{KellerPinchoverPogorzelski2020}.  In the
present setting the remainder is positive at every vertex, a stronger
property than is required by this definition.

\begin{corollary}\label{cor:subcritical}
For $N=3,4$, let $h_N$ be the supersolution in Lemma~\ref{lem:N3} or
Lemma~\ref{lem:N4}, respectively, and define
\begin{equation}\label{eq:positive-remainder}
 V_N(n):=\frac{(-\lap h_N)(n)}{h_N(n)}-\frac{A_N}{|n|^2},
 \qquad n\in X_N.
\end{equation}
Then $V_N(n)>0$ for every $n\in X_N$, and
\begin{equation}\label{eq:remainder-Hardy}
 \cE_N[u]-A_N\cD_N[u]
 \ge\sum_{n\in X_N}V_N(n)|u(n)|^2
\end{equation}
for every finitely supported $u$ with $u(0)=0$.
Thus the form on the left is subcritical, although $C(N)=A_N$.
In particular, $A_N|n|^{-2}$ is not a pointwise optimal Hardy weight.
\end{corollary}

\begin{proof}
Positivity of $V_N$ is exactly the strict inequality in
Lemma~\ref{lem:N3} or Lemma~\ref{lem:N4}.  Substituting
$(-\lap h_N)/h_N=A_N|n|^{-2}+V_N$ in \eqref{eq:gsr} and discarding its
nonnegative transformed edge sum proves \eqref{eq:remainder-Hardy}.
This proves subcriticality and pointwise improvability by their definitions.
There is no contradiction with scalar sharpness: if
$V_N(n)\ge\varepsilon|n|^{-2}$ held for some $\varepsilon>0$ at every
vertex, then \eqref{eq:remainder-Hardy} would imply
$C(N)\ge A_N+\varepsilon$, contrary to Proposition~\ref{prop:upper}.
\end{proof}

The remainder in Corollary~\ref{cor:subcritical} is therefore a pointwise
improvement, not an optimal remainder.  Its relation with the endpoint
construction is worth noting: at $\beta=\beta_N^*$ the $s^{-2}$ correction
in the axial expansion \eqref{eq:affine-expansion} vanishes.  Thus the
supersolution is tuned precisely so that its leading asymptotic potential
remains $A_N|n|^{-2}$ while a strictly positive finite-scale remainder
survives.  This is the mechanism behind the coexistence of scalar sharpness
and subcriticality; no criticality or optimality claim for $V_N$ is made.

\section{A Gaussian Ritz construction in dimension nine}

The dimension-nine trial space is built from lattice Gaussians.  Its
motivation comes from the Gamma--Laplace representation of the formal
continuum Hardy ground state, while its analysis uses tensor factorization
and the Jacobi theta transformation.  After the scales are fixed, all bounds
needed for the strict inequality are one-sided and exact.

For real-valued functions $f,g$ on $X_N$, extended by zero to the origin, we
use the polarizations
\begin{align*}
 \cE_N(f,g)&:=\sum_{j=1}^N\sum_{n\in\Z^N}
 (f(n+e_j)-f(n))(g(n+e_j)-g(n)),\\
 \cD_N(f,g)&:=\sum_{n\in X_N}\frac{f(n)g(n)}{|n|^2}.
\end{align*}
Thus $\cE_N[f]=\cE_N(f,f)$ and $\cD_N[f]=\cD_N(f,f)$.

\subsection{Admissibility, theta kernels, and Gaussian factorization}

For $N\ge1$ and $0<q<1$, define
\begin{equation}\label{eq:gq}
 g_{q,N}(0)=0,\qquad
 g_{q,N}(n)=q^{|n|^2}\quad(n\in\Z^N\setminus\{0\}).
\end{equation}
The following approximation fact will be used whenever Gaussian functions are
inserted into a variational problem initially defined on finitely supported
functions.

\begin{lemma}\label{lem:form-domain}
Let $N\ge1$ and $0<q<1$.  Then $g_{q,N}\in\ell^2(X_N)$ and both
$\cE_N[g_{q,N}]$ and $\cD_N[g_{q,N}]$ are finite.  If
\[
 g_{q,N,R}:=g_{q,N}\ind_{\{\|n\|_\infty\le R\}},
\]
then
\[
 \|g_{q,N}-g_{q,N,R}\|_2\longrightarrow0,
\]
\[
 \cE_N[g_{q,N}-g_{q,N,R}]\longrightarrow0,
 \qquad
 \cD_N[g_{q,N}-g_{q,N,R}]\longrightarrow0.
\]
The same conclusions hold for every finite linear combination of Gaussian
profiles $g_{q,N}$.
\end{lemma}

\begin{proof}
The product structure gives
\[
 \sum_{n\in\Z^N}q^{2|n|^2}
 =\left(\sum_{k\in\Z}q^{2k^2}\right)^N<\infty.
\]
For any $f\in\ell^2(X_N)$ extended by zero to the origin,
$|z-w|^2\le2|z|^2+2|w|^2$ implies
\begin{equation}\label{eq:bounded-forms}
 \cE_N[f]\le4N\|f\|_2^2,
 \qquad
 \cD_N[f]\le\|f\|_2^2.
\end{equation}
Indeed, translation preserves the full-lattice $\ell^2$ norm, and
$|n|^{-2}\le1$ on $X_N$.  Since the $\ell^2$ tail of $g_{q,N}$ outside a
growing cube tends to zero, \eqref{eq:bounded-forms} gives convergence of
both quadratic forms on the truncation error.  Finite linear combinations are
treated in exactly the same way.  Finally, for either form
$Q=\cE_N,\cD_N$,
\[
 \bigl|\sqrt{Q[f]}-\sqrt{Q[f_R]}\bigr|
 \le\sqrt{Q[f-f_R]},
\]
which follows from the triangle inequality for the corresponding difference
or weighted-value sequence in $\ell^2$.  Hence the values of the forms
converge as well.
\end{proof}

For $0<x<1$, let
\begin{equation}\label{eq:theta}
 \vartheta(x):=\sum_{k\in\Z}x^{k^2}
 =1+2\sum_{k=1}^{\infty}x^{k^2}.
\end{equation}
This is the Jacobi theta function $\theta_3$ at zero.  Its imaginary
transformation, equivalently Poisson summation for a Gaussian, gives
\begin{equation}\label{eq:theta-modular}
 \vartheta(e^{-s})
 =\sqrt{\frac{\pi}{s}}\,\vartheta(e^{-\pi^2/s}),
 \qquad s>0;
\end{equation}
see \cite[Eq.~(20.7.32)]{DLMF}.  In particular,
\begin{equation}\label{eq:theta-modular-lower}
 \vartheta(e^{-s})\ge \sqrt{\frac{\pi}{s}},\qquad s>0,
\end{equation}
by retaining the zero term in the transformed theta series.

For $0<a,b<1$, define
\begin{equation}\label{eq:G-series}
 G(a,b):=\sum_{j=0}^{\infty}
 (ab)^{j^2}(1-a^{2j+1})(1-b^{2j+1}).
\end{equation}
Absolute convergence follows by comparison with $\vartheta(ab)$.

A second kernel describes all Gaussian Hardy denominators in arbitrary
dimension:
\begin{equation}\label{eq:FN-series}
 F_N(x):=\sum_{n\in\Z^N\setminus\{0\}}
          \frac{x^{|n|^2}}{|n|^2},\qquad 0<x<1.
\end{equation}

\begin{lemma}\label{lem:FN-kernel}
For every $N\ge1$ and $0<x<1$,
\begin{align}
 F_N(x)
 &=\int_{-\log x}^{\infty}
      \bigl(\vartheta(e^{-s})^N-1\bigr)\,\dd s,
 \label{eq:FN-large-integral}\\
 &=\int_0^x\frac{\vartheta(t)^N-1}{t}\,\dd t.
 \label{eq:FN-small-integral}
\end{align}
Moreover,
\begin{equation}\label{eq:DN-gaussian-kernel}
 \cD_N(g_{a,N},g_{b,N})=F_N(ab),\qquad 0<a,b<1.
\end{equation}
\end{lemma}

\begin{proof}
All summands are nonnegative, so Tonelli's theorem applies.  Since
\[
 \int_{-\log x}^{\infty}e^{-s|n|^2}\,\dd s
 =\frac{x^{|n|^2}}{|n|^2},
\]
summing first over $n\ne0$ gives \eqref{eq:FN-large-integral} because
\[
 \sum_{n\in\Z^N}e^{-s|n|^2}=\vartheta(e^{-s})^N.
\]
The substitution $t=e^{-s}$ gives \eqref{eq:FN-small-integral}.  Finally,
\[
 \cD_N(g_{a,N},g_{b,N})
 =\sum_{n\ne0}\frac{(ab)^{|n|^2}}{|n|^2}=F_N(ab).
\]
\end{proof}

A useful consequence of \eqref{eq:FN-small-integral}, valid in every
dimension, is obtained from the elementary theta bound
$\vartheta(t)\ge1+2t+2t^4$:
\begin{equation}\label{eq:FN-polynomial-general}
 F_N(x)\ge
 \int_0^x\frac{(1+2t+2t^4)^N-1}{t}\,\dd t,
 \qquad 0<x<1.
\end{equation}
The right-hand side is an explicit polynomial with positive coefficients.

The energy admits an equally explicit tensor factorization.

\begin{lemma}\label{lem:theta-energy}
For $N\ge1$ and $0<a,b<1$,
\begin{equation}\label{eq:theta-energy-general}
 \cE_N(g_{a,N},g_{b,N})
 =2N\,G(a,b)\vartheta(ab)^{N-1}+2N(a+b-1).
\end{equation}
\end{lemma}

\begin{proof}
First consider the full-lattice Gaussians
$\widetilde g_{a,N}(n)=a^{|n|^2}$ and
$\widetilde g_{b,N}(n)=b^{|n|^2}$, whose values at the origin are one.
Fix a coordinate $\ell$.  Writing $n=(n_1,\ldots,n_N)$,
\[
 \widetilde g_{a,N}(n+e_\ell)-\widetilde g_{a,N}(n)
 =a^{\sum_{r\ne\ell}n_r^2}
   \bigl(a^{(n_\ell+1)^2}-a^{n_\ell^2}\bigr),
\]
and analogously for $b$.  Absolute convergence and Fubini's theorem therefore
give
\begin{align*}
 &\sum_{n\in\Z^N}
 \bigl(\widetilde g_{a,N}(n+e_\ell)-\widetilde g_{a,N}(n)\bigr)
 \bigl(\widetilde g_{b,N}(n+e_\ell)-\widetilde g_{b,N}(n)\bigr)\\
 &\qquad=\vartheta(ab)^{N-1}S(a,b),
\end{align*}
where
\[
 S(a,b):=\sum_{k\in\Z}
 \bigl(a^{(k+1)^2}-a^{k^2}\bigr)
 \bigl(b^{(k+1)^2}-b^{k^2}\bigr).
\]
For $j\ge0$, the summand at $k=j$ equals
\[
 (ab)^{j^2}(1-a^{2j+1})(1-b^{2j+1}),
\]
and the summand at $k=-j-1$ has the same value.  Hence
\[
 S(a,b)=2G(a,b).
\]
Summing over the $N$ coordinate directions yields
\begin{equation}\label{eq:full-gaussian-energy}
 \cE_N(\widetilde g_{a,N},\widetilde g_{b,N})
 =2N\,G(a,b)\vartheta(ab)^{N-1}.
\end{equation}

We now pass from the full-lattice Gaussians to the admissible profiles
$g_{a,N},g_{b,N}$, which vanish at the origin by definition.  Exactly $2N$
unordered nearest-neighbour edges are incident with zero.  On each such edge
the bilinear contribution for the full-lattice Gaussians is
$(a-1)(b-1)$, whereas for $g_{a,N},g_{b,N}$ it is $ab$.  The correction per
edge is therefore $a+b-1$, and all other edges are unchanged.  Adding these
$2N$ corrections to \eqref{eq:full-gaussian-energy} proves
\eqref{eq:theta-energy-general}.
\end{proof}

For the remainder of Section~4 we specialize to $N=9$ and abbreviate
\[
 g_q:=g_{q,9},\qquad F:=F_9.
\]
Thus
\begin{equation}\label{eq:N9-F-series}
 F(x)=\sum_{n\in\Z^9\setminus\{0\}}\frac{x^{|n|^2}}{|n|^2},
\end{equation}
and \eqref{eq:theta-energy-general} becomes
\begin{equation}\label{eq:theta-energy}
 \cE_9(g_a,g_b)=18G(a,b)\vartheta(ab)^8+18(a+b-1).
\end{equation}

\subsection{Gaussian Ritz spaces and the critical radial profile}

The link with Gaussian functions is not specific to dimension nine.  For
$N\ge3$, set
\[
 \alpha_N:=\frac{N-2}{2},\qquad \nu_N:=\frac{N-2}{4}.
\]
The formal continuum Hardy ground state satisfies
\begin{equation}\label{eq:continuum-ground-general}
 -\Delta_{\R^N}|x|^{-\alpha_N}
 =A_N|x|^{-2}|x|^{-\alpha_N}.
\end{equation}
In the squared radial variable $s=|x|^2$ this profile is $s^{-\nu_N}$, and
the Gamma integral gives
\begin{equation}\label{eq:gamma-gaussian-superposition-general}
 s^{-\nu_N}
 =\frac1{\Gamma(\nu_N)}
   \int_0^\infty t^{\nu_N-1}e^{-st}\,\dd t,
 \qquad s>0;
\end{equation}
see \cite[Eq.~(5.9.1)]{DLMF}.  Thus the radial Gaussians
$e^{-t|n|^2}=q^{|n|^2}$, $q=e^{-t}$, are the Laplace building blocks of the
continuum critical power in every dimension $N\ge3$.

The preceding kernels give a finite-dimensional variational reduction for
arbitrary Gaussian families.

\begin{proposition}\label{prop:gaussian-Ritz-general}
Let $N\ge1$, let $m\ge1$, and let $q_1,\ldots,q_m\in(0,1)$ be distinct.
Set
\[
 V_N(q):=\operatorname{span}\{g_{q_1,N},\ldots,g_{q_m,N}\}.
\]
Define the symmetric matrices
\begin{align}
 \mathsf E_N(q)_{ij}
 &=2N\,G(q_i,q_j)\vartheta(q_iq_j)^{N-1}
   +2N(q_i+q_j-1),\label{eq:general-E-matrix}\\
 \mathsf D_N(q)_{ij}
 &=F_N(q_iq_j).\label{eq:general-D-matrix}
\end{align}
Then $\mathsf D_N(q)$ is positive definite and
\begin{equation}\label{eq:general-Ritz-value}
 \min_{0\ne v\in V_N(q)}\frac{\cE_N[v]}{\cD_N[v]}
 =\min_{c\ne0}\frac{c^T\mathsf E_N(q)c}{c^T\mathsf D_N(q)c}.
\end{equation}
Equivalently, this minimum is the smallest generalized eigenvalue of
\begin{equation}\label{eq:generalized-eigenproblem-general}
 \mathsf E_N(q)c=\lambda\mathsf D_N(q)c.
\end{equation}
\end{proposition}

\begin{proof}
Lemma~\ref{lem:theta-energy} and \eqref{eq:DN-gaussian-kernel} give the two
Gram matrices in \eqref{eq:general-E-matrix}--\eqref{eq:general-D-matrix}.
The functions $g_{q_i,N}$ are linearly independent.  Indeed, after ordering
$0<q_1<\cdots<q_m<1$, restrict a linear relation to the axis $ke_1$, divide
by $q_m^{k^2}$, and let $k\to\infty$ to obtain the last coefficient equal to
zero; iteration gives all coefficients zero.  Since $\cD_N$ is a positive
inner product on nonzero functions on $X_N$, $\mathsf D_N(q)$ is positive
definite.  Formula \eqref{eq:general-Ritz-value} is then the usual
Rayleigh--Ritz principle, and \eqref{eq:generalized-eigenproblem-general} is
its generalized eigenvalue formulation.
\end{proof}

We now take $N=9$ and $m=3$.  In this case $\nu_9=7/4$ and
$A_9=49/4$.  The Laplace integral
\eqref{eq:gamma-gaussian-superposition-general} is naturally sampled on a
logarithmic $t$-scale: separated values of $t$ resolve different radial
ranges while preserving the same Gaussian tensor structure.  We use three
such scales and choose rational $q=e^{-t}$ values so that all subsequent
estimates can be carried out exactly:
\begin{equation}\label{eq:N9-scales}
 q_1=\frac7{25},\qquad q_2=\frac{18}{25},\qquad q_3=\frac{23}{25}.
\end{equation}
They lie close to a geometric logarithmic mesh:
$q_2\approx q_3^4$ and $q_1\approx q_2^4$, equivalently
$t_2\approx4t_3$ and $t_1\approx4t_2$ for $t_i=-\log q_i$.  The particular rational values
are used only to make the verification below exact; the analytic role of the
three basis functions is to probe three separated radial scales in the
Laplace representation.

A reproducible way to arrive at such a trial is the following.  One first
samples the Gamma--Laplace representation \eqref{eq:gamma-gaussian-superposition-general}
on a short geometric mesh in the variable $t$, so that the basis resolves
inner, intermediate and outer radial scales.  For each tentative mesh one
forms the two Gram matrices in Proposition~\ref{prop:gaussian-Ritz-general}
and computes their smallest generalized eigenvalue.  Once a mesh gives a
value below $A_9$, the scales are moved to nearby rational values and the
corresponding minimizing eigenvector is replaced by a simple nearby rational
direction.  The last step is not numerical: the chosen rational mesh and
vector are then certified from scratch by one-sided theta estimates.  The
values in \eqref{eq:N9-scales} and \eqref{eq:N9-parameters} are one convenient
outcome of this strategy, with a comfortable strict gap and particularly
simple exact arithmetic.

Inside the corresponding Ritz space the coefficients are variational
parameters.  To prove a strict upper bound for $C(9)$ it suffices to exhibit
one direction with quotient below $A_9$.  We take the simple positive rational
direction
\begin{equation}\label{eq:N9-parameters}
 (c_1,c_2,c_3)=(256,16,1)=(16^2,16,1).
\end{equation}
Define
\begin{equation}\label{eq:phi-m}
 \phi_m=\sum_{i=1}^3c_iq_i^m\quad(m\ge0),\qquad
 u_9(0)=0,\quad u_9(n)=\phi_{|n|^2}\quad(n\ne0).
\end{equation}
Thus $u_9=\sum_i c_i g_{q_i}$ belongs to the form domain and is approximated
by finitely supported functions by Lemma~\ref{lem:form-domain}.  The next two
subsections prove directly that its Rayleigh quotient is strictly below
$A_9$.

\subsection{An exact upper bound for the energy}

We first record a geometric estimate for the tails of both positive series.
For integers $K,J\ge0$, set
\begin{align}
 \Theta_K(x)&:=1+2\sum_{k=1}^K x^{k^2}
       +\frac{2x^{(K+1)^2}}{1-x^{2K+3}},\label{eq:ThetaK}\\
 G_J(a,b)&:=\sum_{j=0}^J
 (ab)^{j^2}(1-a^{2j+1})(1-b^{2j+1})
       +\frac{(ab)^{(J+1)^2}}{1-(ab)^{2J+3}}.\label{eq:GJ}
\end{align}
These expressions are rational when their arguments are rational.

The tail of $G$ is estimated after dropping the two factors bounded by one.
Its retained terms keep the difference structure of the energy, so this
estimate does not subtract two separately approximated infinite series.

\begin{lemma}\label{lem:theta-majorant}
For $0<x,a,b<1$ and integers $K,J\ge0$,
\[
 0<\vartheta(x)\le\Theta_K(x),\qquad
 0<G(a,b)\le G_J(a,b).
\]
\end{lemma}

\begin{proof}
For every integer $r\ge0$,
\[
 (K+1+r)^2\ge (K+1)^2+(2K+3)r,
\]
because the difference is $r(r-1)\ge0$.  Consequently,
\[
 \sum_{k=K+1}^{\infty}x^{k^2}
 \le x^{(K+1)^2}\sum_{r=0}^{\infty}x^{(2K+3)r}
 =\frac{x^{(K+1)^2}}{1-x^{2K+3}}.
\]
This proves the theta bound.  In the tail of \eqref{eq:G-series},
use $0<(1-a^{2j+1})(1-b^{2j+1})<1$ and apply the same inequality
with $K$ replaced by $J$ and $x$ by $ab$.  Adding the retained terms
proves the second bound.  Positivity is immediate from the defining series.
\end{proof}

For $i\le j$, let $\omega_{ij}=1$ if $i=j$, and $\omega_{ij}=2$ otherwise.
Define the positive tensor contributions
\begin{equation}\label{eq:six-energy-terms}
 T_{ij}:=18\omega_{ij}c_ic_j
            G(q_i,q_j)\vartheta(q_iq_j)^8.
\end{equation}
Bilinearity and \eqref{eq:theta-energy} give
\begin{equation}\label{eq:six-energy-upper}
 \cE_9[u_9]=\sum_{1\le i\le j\le3}T_{ij}-\frac{12869766}{25}.
\end{equation}
Indeed, the origin correction is
\[
 18\sum_{i,j}c_ic_j(q_i+q_j-1)
 =18(2\phi_0\phi_1-\phi_0^2)=-\frac{12869766}{25},
 \qquad \phi_0=273,\quad\phi_1=\frac{2103}{25}.
\]

The next estimate uses the rational expressions \eqref{eq:ThetaK} and
\eqref{eq:GJ} at the six pairs.  The main text records the resulting six
one-sided bounds; the substitutions, eighth-power checks and all elementary
rational arithmetic are collected in Appendix~\ref{app:N9-energy-certificate}.

\begin{lemma}\label{lem:N9-energy}
The function in \eqref{eq:phi-m} satisfies
\begin{equation}\label{eq:N9-energy-bound}
 \cE_9[u_9]<4\,060\,000.
\end{equation}
\end{lemma}

\begin{proof}
Lemma~\ref{lem:theta-majorant} turns each of the six positive tensor
contributions in \eqref{eq:six-energy-terms} into a finite rational upper
bound.  With the truncation orders specified in
Appendix~\ref{app:N9-energy-certificate}, the exact comparisons give
\begin{align*}
 T_{11}&<2\,248\,800,& T_{12}&<741\,700,& T_{13}&<31\,100,\\
 T_{22}&<800\,000,& T_{23}&<348\,300,& T_{33}&<404\,400.
\end{align*}
The appendix contains the substitutions in $\Theta_K$ and $G_J$ and the
integer comparisons verifying every rounding.  Using the exact origin
correction in \eqref{eq:six-energy-upper},
\begin{align*}
 \cE_9[u_9]
 &<2\,248\,800+741\,700+31\,100+800\,000+348\,300+404\,400
       -\frac{12869766}{25}\\
 &=\frac{101\,487\,734}{25}<4\,060\,000.
\end{align*}
\end{proof}

\subsection{An analytic lower bound for the denominator}

The denominator admits an integral representation in which the modular
identity for the theta function provides a useful lower bound.  This
representation separates the three products of small Gaussian parameters
from the three larger products.  A fixed polynomial handles the former;
the integral estimate handles the latter, including their complete tails.

For $N=9$, Lemma~\ref{lem:FN-kernel} gives the two exact representations
\begin{align}
 F(x)&=\int_{-\log x}^{\infty}
            \bigl(\vartheta(e^{-s})^9-1\bigr)\,\dd s,
 \label{eq:N9-F-integral}\\
 F(x)&=\int_0^x\frac{\vartheta(t)^9-1}{t}\,\dd t,
 \label{eq:N9-F-small-integral}
\end{align}
valid for $0<x<1$.  For small $x$, the general minorant \eqref{eq:FN-polynomial-general},
specialized to $N=9$, gives an explicit positive polynomial.  Retaining its
first eight terms and rounding two coefficients downward gives
\begin{equation}\label{eq:N9-p-polynomial}
 p(x):=18x+72x^2+224x^3+508x^4
       +864x^5+1232x^6+1810x^7+2826x^8.
\end{equation}
At the larger products of Gaussian parameters, the modular estimate below is
stronger than this low-order Taylor minorant.

\begin{lemma}\label{lem:N9-denominator}
The function in \eqref{eq:phi-m} satisfies
\begin{equation}\label{eq:N9-denominator-bound}
 \cD_9[u_9]>331\,800.
\end{equation}
\end{lemma}

\begin{proof}
From \eqref{eq:N9-F-small-integral} and
$\vartheta(t)\ge1+2t+2t^4$,
\begin{align*}
 F(x)
 &\ge\int_0^x\frac{(1+2t+2t^4)^9-1}{t}\,\dd t\\
 &=18x+72x^2+224x^3+\frac{1017}{2}x^4
   +864x^5+1232x^6+\frac{12672}{7}x^7+2826x^8
   +P_{>8}(x),
\end{align*}
where $P_{>8}$ has nonnegative coefficients and no terms of degree at most
eight.  Comparing coefficients with the polynomial
\eqref{eq:N9-p-polynomial} gives
\begin{equation}\label{eq:N9-F-polynomial}
 F(x)>p(x)\qquad(0<x<1).
\end{equation}
This polynomial is used for the three smaller products
$q_1^2,q_1q_2,q_1q_3$.

For the three larger products, write $t=-\log x$.  Splitting
\eqref{eq:N9-F-integral} at $s=1$ and applying the Jacobi lower bound
\eqref{eq:theta-modular-lower} gives, for $0<t<1$,
\begin{equation}\label{eq:N9-F-Poisson}
 F(x)>45+t+\frac27\pi^{9/2}(t^{-7/2}-1).
\end{equation}
Thus the complete tails of the denominator kernel are controlled analytically,
not truncated.  To turn \eqref{eq:N9-F-polynomial} and
\eqref{eq:N9-F-Poisson} into a certificate involving rational arithmetic
only, use rational lower bounds for $\pi$, the identity
\[
 -\log x=2\sum_{k=0}^{\infty}\frac{y^{2k+1}}{2k+1},
 \qquad y=\frac{1-x}{1+x},
\]
and geometric domination of its tail.  The explicit substitutions are given
in Appendix~\ref{app:N9-denominator-certificate}.  They prove
\begin{align}
 256^2F(49/625)&>130000,\notag\\
 2\cdot256\cdot16\,F(126/625)&>78900,\notag\\
 2\cdot256\,F(161/625)&>8700,\label{eq:N9-small-product-summary}\\
 F(324/625)&>209,\notag\\
 F(414/625)&>1092,\notag\\
 F(529/625)&>25850.\notag
\end{align}
Expanding the square of the trial function and using Tonelli's theorem now
gives
\begin{align*}
 \cD_9[u_9]
 &=256^2F(49/625)+2\cdot256\cdot16\,F(126/625)
       +2\cdot256\,F(161/625)\\
 &\quad+16^2F(324/625)+32F(414/625)+F(529/625)\\
 &>130000+78900+8700+256\cdot209+32\cdot1092+25850\\
 &=331898>331800.
\end{align*}
\end{proof}

The resulting quotient estimate is strict.  Consequently, the
strict inequality survives the finite-support approximation established
in Lemma~\ref{lem:form-domain}.

The argument needs only one function whose quotient is smaller than $A_9$;
it does not identify the optimal Gaussian profile or the exact value of
$C(9)$.

\begin{proposition}\label{prop:N9}
One has $C(9)<49/4$.
\end{proposition}

\begin{proof}
By Lemmas~\ref{lem:N9-energy} and \ref{lem:N9-denominator},
\[
 \frac{\cE_9[u_9]}{\cD_9[u_9]}
 <\frac{4\,060\,000}{331\,800}<\frac{49}{4},
\]
because
\[
 4\cdot4\,060\,000=16\,240\,000
 <16\,258\,200=49\cdot331\,800.
\]
The final comparison has the positive integer gap $18\,200$.
The denominator is positive because $u_9$ is nonzero.
Lemma~\ref{lem:form-domain} shows that the cube truncations $u_{9,R}$ have
both quadratic forms converging to those of $u_9$.  Their denominators are
positive for $R\ge1$, so their quotients converge as well.
For all sufficiently large $R$, the quotient is strictly below $49/4$.
These truncations belong to the class in \eqref{eq:def-CN}, which proves
the assertion.
\end{proof}

\section{Orbit compressions and the coordinate Ritz ladder}\label{sec:high-dimensional-orbits}

Let $R$ be multiplication by $|n|$ on $\Z^N\setminus\{0\}$.  For finitely
supported $v$, define
\begin{equation}\label{eq:KN}
 (K_Nv)(n):=2N|n|^2v(n)-\sum_{m\sim n}|n|\,|m|v(m).
\end{equation}
If $u=Rv$, expansion of the edge squares gives
\begin{equation}\label{eq:conjugation}
 \cE_N[u]=\langle K_Nv,v\rangle_{\ell^2},
 \qquad
 \cD_N[u]=\norm{v}_2^2.
\end{equation}
Consequently, every finite-dimensional Rayleigh quotient of $K_N$ gives an
upper bound for $C(N)$.  The purpose of this section is not merely to produce
another family of trial functions.  The nested orbit compressions form a
Ritz ladder whose successive crossings of $A_N$ identify how many coordinate
shells are needed in dimensions twelve, eleven and ten; the fact that the
entire ladder stays above $A_9$ then isolates dimension nine as the first
dimension in which this coordinate mechanism provably cannot certify the
strict inequality.

The hyperoctahedral group
\[
 B_N=(\mathbb Z_2)^N\rtimes S_N
\]
acts on $\Z^N$ by independent sign changes and coordinate permutations and
preserves both the lattice graph and $|n|$.  Its orbits are indexed by the
coordinate multiplicities from \eqref{eq:coordinate-multiplicities}.  If an
orbit has type $a=(a_0,a_1,\ldots)$, then
\begin{equation}\label{eq:general-orbit-size}
 |\cO_a|=2^{N-a_0}\frac{N!}{\prod_{r\ge0}a_r!}.
\end{equation}
The orbit partition is equitable; see \cite[Section~9.3]{GodsilRoyle2001}.
If two orbits $\cO_a,\cO_b$ are joined with constant edge weight $w_{ab}$,
and a vertex in $\cO_a$ has $b_{ab}$ neighbours in $\cO_b$ while a vertex
in $\cO_b$ has $b_{ba}$ neighbours in $\cO_a$, then for the normalized orbit
indicators $\psi_a=|\cO_a|^{-1/2}\ind_{\cO_a}$,
\begin{equation}\label{eq:normalized-quotient-entry}
 \langle\psi_a,K_N\psi_b\rangle
 =w_{ab}\sqrt{b_{ab}b_{ba}}.
\end{equation}
This follows from $|\cO_a|b_{ab}=|\cO_b|b_{ba}$.

\subsection{The coordinate ladder}

We now select the nested sequence of coordinate orbits
\[
 \cO_{1^k}:=\{n\in\Z^N:n_j\in\{0,\pm1\},\ |n|^2=k\},
 \qquad 1\le k\le N,
\]
and write
\[
 \psi_k:=|\cO_{1^k}|^{-1/2}\ind_{\cO_{1^k}},
 \qquad
 V_k(N):=\operatorname{span}\{\psi_1,\ldots,\psi_k\}.
\]
Formula \eqref{eq:general-orbit-size} gives
\begin{equation}\label{eq:orbit-size}
 |\cO_{1^k}|=2^k\binom Nk.
\end{equation}
Moreover, the only adjacencies inside this chain are between successive
orbits, with transition numbers
\begin{equation}\label{eq:orbit-transition-numbers}
 b_{k,k+1}=2(N-k),\qquad b_{k+1,k}=k+1.
\end{equation}

\begin{lemma}\label{lem:orbit-matrix}
For $1\le k\le N$, the compression of $K_N$ to $V_k(N)$ is the symmetric
tridiagonal matrix $M_N^{(k)}$ with
\begin{equation}\label{eq:orbit-matrix}
 (M_N^{(k)})_{jj}=2Nj,
 \qquad
 (M_N^{(k)})_{j,j+1}=(M_N^{(k)})_{j+1,j}
 =-(j+1)\sqrt{2j(N-j)}
\end{equation}
for $1\le j<k$.
\end{lemma}

\begin{proof}
The diagonal term of $K_N$ is $2Nj$ on $\cO_{1^j}$.  Between
$\cO_{1^j}$ and $\cO_{1^{j+1}}$ the edge weight in \eqref{eq:KN} is
$-\sqrt{j(j+1)}$.  Combining \eqref{eq:normalized-quotient-entry} with
\eqref{eq:orbit-transition-numbers} gives
\[
 \langle\psi_j,K_N\psi_{j+1}\rangle
 =-\sqrt{j(j+1)}\sqrt{2(N-j)(j+1)}
 =-(j+1)\sqrt{2j(N-j)}.
\]
Nonconsecutive coordinate orbits are not adjacent, which proves the claim.
\end{proof}

The nested matrices in Lemma~\ref{lem:orbit-matrix} define a natural Ritz
ladder.

\begin{definition}\label{def:Lambda-k}
For $1\le k\le N$, set
\begin{equation}\label{eq:def-Lambda-k}
 \Lambda_k(N):=\lambda_{\min}\!\left(M_N^{(k)}\right).
\end{equation}
\end{definition}

\begin{proposition}\label{prop:Ritz-ladder}
For $N\ge3$ and $1\le k\le N$,
\begin{equation}\label{eq:Ritz-upper}
 C(N)\le \Lambda_k(N).
\end{equation}
For $1\le k<N$ the ladder is strictly decreasing:
\begin{equation}\label{eq:Ritz-strict-monotonicity}
 \Lambda_{k+1}(N)<\Lambda_k(N).
\end{equation}
If
\[
 P_k(\lambda;N):=\det\!\left(M_N^{(k)}-\lambda I\right),
 \qquad P_0(\lambda;N):=1,
\]
then
\begin{align}
 P_1(\lambda;N)&=2N-\lambda,\label{eq:P1}\\
 P_k(\lambda;N)
 &=(2Nk-\lambda)P_{k-1}(\lambda;N)
   -2k^2(k-1)(N-k+1)P_{k-2}(\lambda;N)
 \label{eq:Pk-recurrence}
\end{align}
for $2\le k\le N$.
\end{proposition}

\begin{proof}
The variational estimate \eqref{eq:Ritz-upper} follows from
\eqref{eq:conjugation}, since every vector in $V_k(N)$ has finite support.
Because $M_N^{(k)}$ is the leading principal submatrix of
$M_N^{(k+1)}$, the Courant--Fischer min--max principle
\cite[Section~4.2]{HornJohnson2013} gives
$\Lambda_{k+1}(N)\le\Lambda_k(N)$.

To see that the inequality is strict, let $x$ be a normalized eigenvector of
$M_N^{(k)}$ for $\Lambda_k(N)$.  Its last coordinate cannot vanish: if
$x_k=0$, the last row of the tridiagonal eigenvalue equation and the nonzero
off-diagonal entries force successively $x_{k-1}=\cdots=x_1=0$.  If equality
held in the min--max inequality, $(x,0)$ would attain the lowest Rayleigh
quotient of $M_N^{(k+1)}$ and hence would be an eigenvector.  The last row of
that eigenvalue equation would then give
\[
 -(k+1)\sqrt{2k(N-k)}\,x_k=0,
\]
a contradiction.  This proves \eqref{eq:Ritz-strict-monotonicity}.

Finally, \eqref{eq:Pk-recurrence} is the standard determinant recurrence for
a tridiagonal matrix, since the square of the last off-diagonal entry is
$2k^2(k-1)(N-k+1)$.
\end{proof}

\subsection{The first three levels and their dimensional crossings}

The first three members of the ladder already describe the change from the
large-dimensional one-shell regime to the exceptional dimension nine.  The
third level remains explicit after the characteristic cubic is reduced to
depressed form.

\begin{proposition}\label{prop:coordinate-crossing}
For $N\ge3$, the first three Ritz levels are
\begin{equation}\label{eq:Lambda1-Lambda2}
 \Lambda_1(N)=2N,
 \qquad
 \Lambda_2(N)=3N-\sqrt{N^2+8N-8}.
\end{equation}
Set
\begin{equation}\label{eq:RN-phiN}
 R_N:=\sqrt{N^2+11N-20},
 \qquad
 \phi_N:=\frac13\arccos\!\left(
 \frac{3\sqrt3\,N(7N-16)}{2R_N^3}\right).
\end{equation}
Then
\begin{equation}\label{eq:Lambda3-explicit}
 \Lambda_3(N)
 =4N+\frac{4R_N}{\sqrt3}
 \cos\!\left(\phi_N+\frac{2\pi}{3}\right).
\end{equation}
Equivalently, $\Lambda_3(N)$ is the smallest zero of
\begin{equation}\label{eq:Lambda3-cubic}
 \lambda^3-12N\lambda^2
 +(44N^2-44N+80)\lambda
 -48N^3+120N^2-192N=0.
\end{equation}
Moreover,
\begin{equation}\label{eq:Lambda2-less-2N}
 \Lambda_2(N)<2N\qquad(N\ge3),
\end{equation}
and the first crossings of the continuum coefficient along this ladder are
\begin{align}
 \Lambda_1(N)&<A_N &&(N\ge12),\label{eq:Lambda1-crossing}\\
 \Lambda_2(10)&>A_{10},\qquad
 \Lambda_2(N)<A_N &&(N\ge11),\label{eq:Lambda2-crossing}\\
 \Lambda_3(10)&<A_{10}.\label{eq:Lambda3-crossing}
\end{align}
Consequently,
\begin{equation}\label{eq:Lambda2-main-range}
 C(N)\le\Lambda_2(N)<\min\{A_N,2N\}
 \qquad(N\ge11),
\end{equation}
and $C(10)<A_{10}$ follows from the third level.  Finally,
\begin{equation}\label{eq:Lambda3-asymptotic}
 \Lambda_3(N)
 =2N-4-\frac6N+\frac{51}{N^2}+O(N^{-3})
 \qquad(N\to\infty).
\end{equation}
\end{proposition}

\begin{proof}
Formula \eqref{eq:Lambda1-Lambda2} follows by diagonalizing the first two
matrices in Lemma~\ref{lem:orbit-matrix}.  The radicand in the second formula
is positive for $N\ge3$, and
$N^2+8N-8>N^2$ for $N>1$, which gives
\eqref{eq:Lambda2-less-2N}.

For the third level, expansion of
$\det(M_N^{(3)}-\lambda I)$ gives
\eqref{eq:Lambda3-cubic}.  With $\lambda=4N+y$, this becomes
\begin{equation}\label{eq:Lambda3-depressed}
 y^3-4(N^2+11N-20)y-8N(7N-16)=0.
\end{equation}
Since $M_N^{(3)}$ is real symmetric, the three zeros are real.  The
trigonometric solution of \eqref{eq:Lambda3-depressed} therefore gives
\eqref{eq:Lambda3-explicit}; for the angle in \eqref{eq:RN-phiN}, the shift
by $2\pi/3$ selects the smallest of the three zeros.

For the crossings, the first assertion is equivalent to
\[
 2N<\frac{(N-2)^2}{4}
 \quad\Longleftrightarrow\quad
 N^2-12N+4>0,
\]
whose roots are $6\pm4\sqrt2$.

For the second assertion, extend
$A_x=(x-2)^2/4$ and
$\Lambda_2(x)=3x-\sqrt{x^2+8x-8}$ to real $x\ge8$, and put
$d(x)=A_x-\Lambda_2(x)$.  Then
\[
 d'(x)=\frac{x-8}{2}
       +\frac{x+4}{\sqrt{x^2+8x-8}}>0
 \qquad(x\ge8).
\]
Moreover,
\[
 d(10)=\sqrt{172}-14<0,
 \qquad
 d(11)=\sqrt{201}-\frac{51}{4}>0.
\]
Hence the integer crossing occurs between $10$ and $11$, and
\eqref{eq:Lambda2-less-2N} yields \eqref{eq:Lambda2-main-range}.

For $N=10$,
\[
 M_{10}^{(3)}=
 \begin{pmatrix}
 20&-6\sqrt2&0\\
 -6\sqrt2&40&-12\sqrt2\\
 0&-12\sqrt2&60
 \end{pmatrix},
\]
and at $A_{10}=16$,
\begin{equation}\label{eq:N10-determinant}
 \det\!\left(M_{10}^{(3)}-16I\right)=-96<0.
\end{equation}
Thus $M_{10}^{(3)}-A_{10}I$ is not positive semidefinite, and hence
\eqref{eq:Lambda3-crossing} follows.

Finally, substitute
\[
 \lambda=2N-4+\frac{a}{N}+\frac{b}{N^2}+O(N^{-3})
\]
into \eqref{eq:Lambda3-cubic}.  Comparison of the coefficients of $N$ and
$N^0$ gives $a=-6$ and $b=51$, proving
\eqref{eq:Lambda3-asymptotic}.
\end{proof}

\subsection{Why the coordinate ladder does not reach dimension nine}

The preceding pattern might suggest that one more coordinate orbit should
settle the next dimension.  Dimension nine behaves differently: even the
complete coordinate ladder remains strictly above the continuum coefficient.
This can be proved directly from the determinant recurrence.

\begin{proposition}\label{prop:N9-coordinate-obstruction}
For $N=9$,
\begin{equation}\label{eq:N9-all-coordinate-above}
 \Lambda_k(9)>A_9=\frac{49}{4}
 \qquad(1\le k\le9).
\end{equation}
Equivalently,
\[
 M_9^{(9)}-\frac{49}{4}I
\]
is positive definite.
\end{proposition}

\begin{proof}
Let
\[
 D_k:=\det\!\left(M_9^{(k)}-\frac{49}{4}I\right),
 \qquad D_0:=1.
\]
By \eqref{eq:Pk-recurrence},
\begin{equation}\label{eq:N9-D-recurrence}
 D_k=\left(18k-\frac{49}{4}\right)D_{k-1}
     -2k^2(k-1)(10-k)D_{k-2}
 \qquad(2\le k\le9),
\end{equation}
with $D_1=23/4>0$.
As long as $D_{k-1}>0$, define
\[
 r_k:=4\frac{D_k}{D_{k-1}}.
\]
Then $r_1=23$ and \eqref{eq:N9-D-recurrence} becomes
\begin{equation}\label{eq:N9-r-recurrence}
 r_k=72k-49-
 \frac{32k^2(k-1)(10-k)}{r_{k-1}}.
\end{equation}
Starting from $r_1=23>20$, the recurrence gives successively the exact
lower bounds
\begin{align*}
 r_2&>95-\frac{1024}{20}=\frac{219}{5}>40,\\
 r_3&>167-\frac{4032}{40}=\frac{331}{5}>60,\\
 r_4&>239-\frac{9216}{60}=\frac{427}{5}>80,\\
 r_5&>311-\frac{16000}{80}=111>100,\\
 r_6&>383-\frac{23040}{100}=\frac{763}{5}>150,\\
 r_7&>455-\frac{28224}{150}=\frac{6671}{25}>250,\\
 r_8&>527-\frac{28672}{250}=\frac{51539}{125}>400,\\
 r_9&>599-\frac{20736}{400}=\frac{13679}{25}>500.
\end{align*}
Thus $r_k>0$ and hence $D_k>0$ for every $1\le k\le9$.  Sylvester's criterion \cite[Section~7.2]{HornJohnson2013} gives
\[
 M_9^{(9)}-\frac{49}{4}I>0.
\]
Every $M_9^{(k)}$ is a leading principal submatrix of $M_9^{(9)}$, so the
same positivity holds for $1\le k\le9$.  This is exactly
\eqref{eq:N9-all-coordinate-above}.
\end{proof}

Proposition~\ref{prop:N9-coordinate-obstruction} explains the change of
mechanism in the first dimension for which the present paper proves strictness
beyond the coordinate ladder.  The ladder succeeds successively in dimensions
twelve, eleven and ten, but exhausting it does not cross $A_9$.  The Gaussian
construction of Section~4 therefore supplies genuinely new radial information
rather than merely another coordinate shell.
We now complete the proof of Theorem~\ref{thm:main}.

\begin{proof}[Proof of Theorem~\ref{thm:main}]
Proposition~\ref{prop:upper} gives $C(N)\le A_N$ for every $N\ge3$, and
Proposition~\ref{prop:low-dimensional-values} gives equality in dimensions
three and four.  Proposition~\ref{prop:N9} treats dimension nine, while
Proposition~\ref{prop:coordinate-crossing} gives $C(10)<A_{10}$ and
\[
 C(N)\le\Lambda_2(N)<\min\{A_N,2N\}
 \qquad(N\ge11).
\]
These are precisely the remaining assertions of the theorem.
\end{proof}

\section{A spectral consequence and remaining dimensions}

The sign of the form at $A_N$ has an immediate spectral interpretation.
On $\ell^2(X_N)$, define the Dirichlet lattice Laplacian
\[
 (L_Nu)(n)=2Nu(n)-\sum_{\substack{m\sim n\\m\in X_N}}u(m),
\]
and let $M_N$ denote multiplication by $|n|^{-2}$.
Both are bounded self-adjoint operators, and
$\langle L_Nu,u\rangle=\cE_N[u]$,
$\langle M_Nu,u\rangle=\cD_N[u]$.
Set
\[
 \mathcal H_N:=L_N-A_NM_N.
\]
The potential tends to zero at infinity, so it is a compact multiplication
operator.  Consequently, a negative test value lies below the essential
spectrum and produces a discrete eigenvalue, rather than merely a negative
point in the essential spectrum.

\begin{corollary}\label{cor:spectrum}
For every $N\ge3$,
\[
 \sigma_{\mathrm{ess}}(\mathcal H_N)=[0,4N].
\]
For $N=3,4$, $\mathcal H_N$ is nonnegative.  For every $N\ge9$,
$\mathcal H_N$ has at least one negative isolated eigenvalue of finite
multiplicity.
\end{corollary}

\begin{proof}
On the full lattice, the Fourier transform maps $-\lap$ to multiplication
on the torus by
\[
 2\sum_{j=1}^N(1-\cos\theta_j).
\]
Its essential range is $[0,4N]$: the multiplier is continuous, takes every
value in this interval, and the inverse image of every neighbourhood of such
a value has positive measure.  Thus the full-lattice operator has essential
spectrum $[0,4N]$.

Identify $\ell^2(\Z^N)=\ell^2(X_N)\oplus\mathbb C\ind_{\{0\}}$.
The operator $L_N\oplus0$ differs from the full-lattice Laplacian only in
matrix entries involving zero and its neighbours.  The difference therefore
has finite rank.  Adding or removing the one-dimensional zero summand does
not affect the essential spectrum, so
$\sigma_{\mathrm{ess}}(L_N)=[0,4N]$ by invariance under compact
perturbations.

To see compactness of $M_N$ explicitly, let $M_{N,R}$ be multiplication by
$|n|^{-2}\ind_{\{|n|\le R\}}$.  It has finite rank, and
$\|M_N-M_{N,R}\|\le R^{-2}$ for $R>0$.
Weyl's theorem on invariance of essential spectrum under compact
self-adjoint perturbations now gives the asserted essential spectrum of
$\mathcal H_N$; see \cite{Teschl2014}.

For $N=3,4$, the Hardy inequalities extend from finite support to
$\ell^2(X_N)$ by \eqref{eq:bounded-forms} and density.  Hence
$\mathcal H_N\ge0$.  For $N\ge9$, Theorem~\ref{thm:main} supplies a
finitely supported $u$ with
$\langle\mathcal H_Nu,u\rangle<0$.  The variational principle gives
$\inf\sigma(\mathcal H_N)<0$.  Since the essential spectrum starts at zero,
this spectral minimum is an isolated eigenvalue of finite multiplicity.
\end{proof}

This corollary does not assert uniqueness of the negative eigenvalue or
determine the value of $C(9)$.  It expresses the distinction between
nonnegative forms in dimensions three and four and the existence of a
negative bound state from dimension nine onward.

The theorem leaves a four-dimensional transition window $5\le N\le8$.
Proposition~\ref{prop:shifted-power-fails} shows that the simple shifted-power
mechanism already breaks down at $N=5$, so equality in this range, if it
persists, must arise from a more structured radial or angular mechanism.  It
is therefore natural to ask whether dimension nine is the first dimension in
which the continuum scalar coefficient ceases to be sharp on the lattice.

A natural conjecture is
\[
 C(N)=A_N\quad(3\le N\le8),
 \qquad
 C(N)<A_N\quad(N\ge9).
\]
The substantive open part of this conjecture is the equality
$C(N)=A_N$ for $5\le N\le8$.

\appendix
\section{Exact rational certificate for the dimension-nine Gaussian test}
\label{app:N9-certificate}

This appendix contains only the elementary arithmetic behind the one-sided
bounds used in Section~4.  The analytic estimates---the theta majorant,
Gamma--Laplace representation, Jacobi transformation and polynomial
minorant---remain in the main text.  The purpose here is to make the final
strict comparison independently checkable without decimal evaluation or
computer-assisted interval arithmetic.

\subsection{Energy certificate}\label{app:N9-energy-certificate}

For $(i,j)=(1,1)$, use $K=2$ and $J=1$ in
\eqref{eq:ThetaK}--\eqref{eq:GJ}.  Direct substitution gives
\[
 \Theta_2(q_1^2)
 =1+2q_1^2+2q_1^8+\frac{2q_1^{18}}{1-q_1^{14}}
 <\frac{1157}{1000},
\]
and
\[
 G_1(q_1,q_1)
 =(1-q_1)^2+q_1^2(1-q_1^3)^2
       +\frac{q_1^8}{1-q_1^{10}}<\frac{1187}{2000}.
\]
Since $(1157/1000)^8<803/250$,
\[
 T_{11}<18\cdot256^2\cdot\frac{1187}{2000}\cdot\frac{803}{250}
       <2\,248\,800.
\]

For $(i,j)=(1,2)$, take $K=J=2$.  The corresponding rational comparisons are
\[
 \Theta_2(q_1q_2)<\frac{1407}{1000},\qquad
 G_2(q_1,q_2)<\frac{1633}{5000},\qquad
 \left(\frac{1407}{1000}\right)^8<\frac{77}{5},
\]
and therefore
\[
 T_{12}<36\cdot256\cdot16\cdot\frac{1633}{5000}\cdot\frac{77}{5}
       <741\,700.
\]
For $(i,j)=(1,3)$, again take $K=J=2$:
\[
 \Theta_2(q_1q_3)<\frac{61}{40},\qquad
 G_2(q_1,q_3)<\frac{1149}{10000},\qquad
 \left(\frac{61}{40}\right)^8<\frac{293}{10},
\]
so
\[
 T_{13}<36\cdot256\cdot\frac{1149}{10000}\cdot\frac{293}{10}
       <31\,100.
\]

For $(i,j)=(2,2)$, choose $K=2$, $J=3$:
\[
 \Theta_2(q_2^2)<\frac{2187}{1000},\qquad
 G_3(q_2,q_2)<\frac{3313}{10000},\qquad
 \left(\frac{2187}{1000}\right)^8<524,
\]
which implies
\[
 T_{22}<18\cdot16^2\cdot\frac{3313}{10000}\cdot524<800\,000.
\]
For $(i,j)=(2,3)$, choose $K=2$, $J=4$:
\[
 \Theta_2(q_2q_3)<\frac{1381}{500},\qquad
 G_4(q_2,q_3)<\frac{889}{5000},\qquad
 \left(\frac{1381}{500}\right)^8<3400,
\]
and hence
\[
 T_{23}<36\cdot16\cdot\frac{889}{5000}\cdot3400<348\,300.
\]

Finally, for $(i,j)=(3,3)$, take $K=4$, $J=6$.  The estimates are
\[
 1+2(q_3^2+q_3^8+q_3^{18}+q_3^{32})
       +\frac{2q_3^{50}}{1-q_3^{22}}<\frac{4341}{1000},
\]
and
\[
 \sum_{j=0}^{6}q_3^{2j^2}(1-q_3^{2j+1})^2
       +\frac{q_3^{98}}{1-q_3^{30}}<\frac{89}{500}.
\]
Together with $(4341/1000)^8<126200$, these give
\[
 T_{33}<18\cdot\frac{89}{500}\cdot126200<404\,400.
\]
All comparisons above are between positive rational numbers.  Those
involving eighth powers may be checked by three successive squarings, and
Lemma~\ref{lem:theta-majorant} has already absorbed the infinite tails.

\subsection{Denominator certificate}\label{app:N9-denominator-certificate}

We first record the elementary constants used to rationalize
\eqref{eq:N9-F-Poisson}.  The exponential series gives
\[
 e\le\frac83+\frac{1/24}{1-1/5}=\frac{87}{32}<\frac{11}{4},
\]
while the classical bound $\pi>223/71$ implies
\[
 \pi>\pi_0:=\frac{157}{50}.
\]
Direct lower rounding in \eqref{eq:N9-p-polynomial} gives
\[
 p(4/11)>
 \frac{65+95+107+88+54+28+15+8}{10}=46.
\]
Since $e^{-1}>4/11$, the integral representation and
\eqref{eq:N9-F-polynomial} imply $F(e^{-1})>46$.

For a rational upper bound on $-\log x$, put
$y=(1-x)/(1+x)$.  Then
\begin{equation}\label{eq:app-log-rational}
 -\log x
 =2\sum_{k=0}^{\infty}\frac{y^{2k+1}}{2k+1}
 \le2y+\frac23y^3+\frac{2y^5}{5(1-y^2)}.
\end{equation}
For
\[
 (x_1,x_2,x_3)=\left(\frac{324}{625},\frac{414}{625},\frac{529}{625}\right)
\]
the corresponding values of $y$ are $301/949$, $211/1039$, and $48/577$.
Substitution in \eqref{eq:app-log-rational} gives
\[
 -\log x_1<\frac{329}{500},\qquad
 -\log x_2<\frac{103}{250},\qquad
 -\log x_3<\frac{167}{1000}.
\]

Suppose $t=-\log x<T<1$ and $r^2T<\pi_0$.  From
\eqref{eq:N9-F-Poisson}, the facts $t\ge1-x$ and
$\sqrt{\pi_0}<71/40$ give
\begin{equation}\label{eq:app-B-rational}
 F(x)>46-x+\frac27\pi_0^4
       \left(\frac r{T^3}-\frac{71}{40}\right).
\end{equation}
In each application below the parenthesis is positive; hence the further
lower bound $\pi_0^4>486/5$ preserves the inequality.

For $x_1=324/625$, take $T=329/500$ and $r=273/125$.  Then
\[
 \pi_0-r^2T=\frac{11209}{7812500}>0,
 \qquad \frac r{T^3}>\frac{1533}{200},
\]
and \eqref{eq:app-B-rational} yields
\[
 F(324/625)>209+\frac{242}{4375}>209.
\]
For $x_2=414/625$, take $T=103/250$, $r=69/25$; for
$x_3=529/625$, take $T=167/1000$, $r=542/125$.  The required comparisons are
\[
 \pi_0-\left(\frac{69}{25}\right)^2\frac{103}{250}
   =\frac{121}{78125}>0,
 \qquad
 \frac{69/25}{(103/250)^3}>\frac{7893}{200},
\]
and
\[
 \pi_0-\left(\frac{542}{125}\right)^2\frac{167}{1000}
   =\frac{489}{1953125}>0,
 \qquad
 \frac{542/125}{(167/1000)^3}>\frac{93097}{100}.
\]
Consequently,
\[
 F(414/625)>1092+\frac{187}{4375}>1092,
 \qquad
 F(529/625)>25850+\frac{1979}{8750}>25850.
\]

For the three smaller products, direct substitution in the positive
polynomial \eqref{eq:N9-p-polynomial} gives
\[
 256^2p(49/625)>130000,
\]
\[
 2\cdot256\cdot16\,p(126/625)>78900,
 \qquad
 2\cdot256\,p(161/625)>8700.
\]
For example, retaining the first six positive monomials in the first
inequality and rounding each downward gives
\[
 256^2p(49/625)>92484+29003+7074+1257+167+18=130003.
\]
The other two inequalities are verified by the same direct substitution.
This completes the rational certificate used in
\eqref{eq:N9-small-product-summary}.

\section*{Statements and Declarations}

\noindent\textbf{Funding.}
This work was supported by ANID FONDECYT Grant 1260057.
\medskip

\noindent\textbf{Data Availability.}
No datasets were generated or analysed during the current study.

\medskip
\noindent\textbf{Competing Interests.}
The author has no relevant financial or non-financial interests to disclose.

\end{document}